\documentclass[11pt, oneside]{article} 

\usepackage[utf8]{inputenc}
\usepackage[T1]{fontenc}
\usepackage{xcolor} 
\usepackage{color} 

\usepackage{amsmath}
\usepackage{amssymb} 
\usepackage{mathrsfs} 
\usepackage{bm} 
\usepackage{bbm} 

\usepackage{amsthm} 
\theoremstyle{plain}
\newtheorem{theorem}{Theorem} 

\newtheorem{lemma}[theorem]{Lemma}

\newtheorem{remark}[theorem]{Remark}

\usepackage{algorithm2e}
\SetKwInput{KwInputs}{Inputs}
\usepackage[noend]{algpseudocode}
\makeatletter
\def\BState{\State\hskip-\ALG@thistlm}
\makeatother

\RequirePackage[numbers]{natbib}
\RequirePackage[colorlinks,citecolor=blue,urlcolor=blue]{hyperref}
\usepackage{graphicx} 
\usepackage{subcaption}

\usepackage{tikz} 
\usetikzlibrary{positioning ,shadows,matrix}

\usepackage[english]{babel} 

\usepackage{hyperref} 
\hypersetup{
    pdfpagemode={UseOutlines},
    bookmarksopen,
    pdfstartview={FitH},
    colorlinks,
    linkcolor={blue}, 
    citecolor={blue}, 
    urlcolor={blue} 
}

\usepackage[letterpaper,margin=1.25in]{geometry} 

\usepackage{enumitem} 

\def\iid{\overset{\textnormal{iid}}{\sim}} 

\makeatletter  
\@ifundefined{@currsizeindex} 
  {\RequirePackage{relsize}\let\dolarger\relsize} 
  {\def\dolarger#1{\larger[#1]}} 
\newcommand*\@@bigtimes[2]{\vphantom{\prod} 
  \vcenter{\hbox{\dolarger{4}$\m@th#1\mkern-2mu\times\mkern-2mu$}}} 
\newcommand*\bigtimes{\mathop{\mathpalette\@@bigtimes\relax}\displaylimits} 
\makeatother

\def\N{\mathbb{N}}\def\R{\mathbb{R}}\def\1{\mathbbm{1}}

\def\Lcal{\mathcal{L}}\def\Pcal{\mathcal{P}}\def\Qcal{\mathcal{Q}}\def\Wcal{\mathcal{W}}\def\Zcal{\mathcal{Z}}

\def\cube{[0,1]^d}

\title{\bf Besov-Laplace priors in density estimation: optimal posterior contraction rates and adaptation}
\author{Matteo Giordano \\ \\ University of Turin}
\date{} 

\begin{document}

\maketitle

\abstract{
	Besov priors are nonparametric priors that can model spatially inhomogeneous functions. They are routinely used in inverse problems and imaging, where they exhibit attractive sparsity-promoting and edge-preserving features. A recent line of work has initiated the study of their asymptotic frequentist convergence properties. In the present paper, we consider the theoretical recovery performance of the posterior distributions associated to Besov-Laplace priors in the density estimation model, under the assumption that the observations are generated by a possibly spatially inhomogeneous true density belonging to a Besov space. We improve on existing results and show that carefully tuned Besov-Laplace priors attain optimal posterior contraction rates. Furthermore, we show that  hierarchical procedures involving a hyper-prior on the regularity parameter  lead to adaptation to any smoothness level.
}

\bigskip


\bigskip


\tableofcontents

\section{Introduction}
\label{Sec:Introduction}

Besov priors are a class of probability measures on function spaces constructed via random wavelet expansions, with independently drawn random wavelet coefficients following a distribution with tails between the Gaussian and the Laplace distribution. They were first systematically studied by Lassas et al.~\cite{LSS09} and, over the last two decades, have enjoyed enormous popularity within the inverse problems \cite{LP01,DHS12,KLNS12,L12,R13,BG15,HB15,JPG16,DS17,ABDH18,KLSS23,SS23} and medical imaging \cite{BD06,RVJKLMS06,NSK07,VLS09,SE15} communities.  In the present paper, we focus on `Besov-Laplace' priors (shortly, Laplace priors), corresponding to the case of Laplace-distributed random wavelet coefficients (cf.~Section \ref{Subsec:ObsAndPrior} below for a definition). Among the main advantages, Laplace priors are known to provide an infinite-dimensional, `discretisation-invariant', alternative \cite{LS04,LSS09} to the popular (finite-dimensional) total-variation prior of Rudin et al.~\cite{ROF92}, and to give rise to sparse and edge-preserving reconstruction at the level of the maximum-a-posteriori estimator \cite{KLNS12,SE15,ABDH18}, observed  to perform well in practice in the recovery of inhomogeneous objects with localised sharp irregularities such as images. See \cite{RVJKLMS06,LSS09,DHS12} and references therein. Alongside these qualities,  Besov priors also generally posses a logarithmically concave structure, which facilitates posterior computation \cite{BG15,CDPS18} and theoretical analysis.

	For Laplace priors, the aforementioned features stem from the $\ell^1$-type penalty induced on the wavelet coefficients. This furnishes a Bayesian model for functions in the Besov scale $B^s_{11}, \ s\ge0$,
 in which the local variability is measured in an $L^1$-sense and the smoothness can be described by the $\ell^1$-decay of the wavelet coefficients (cf.~Section \ref{Subsec:Preliminaries} for details). These spaces are known to provide a mathematical characterisation of \textit{spatially inhomogeneous} functions, namely ones that are flat or smooth in some parts of their domain while exhibiting high variation (or even jumps) in other areas. See, e.g., \cite{GN16}, p.348f, for the relationship with the space of bounded variation functions. Laplace priors are thus naturally suited to a number of applications dealing with spatially inhomogeneous objects including, as mentioned above, imaging, but also geophysics, where parameters can exhibit jumps corresponding to changes in layered media, and signal processing, where signals can have extremely localised spikes. In contrast, the widely used Gaussian priors, which induce $\ell^2$-type penalties, model Sobolev-regular functions with less sharp variation, and are known to be unsuited to more structured recovery problems \cite{ADH21,AW23,GRSH22}.

	Despite the popularity in applications, a comprehensive study of the theoretical properties of Besov priors has been initiated only very recently. In particular, we are here interested in the asymptotic recovery performance of the associated posterior distributions in the large sample size setting, under the assumption that the available data is generated by a fixed true function (the `ground truth'). Such \textit{frequentist analysis} of Bayesian nonparametric procedures, and the related theory of \textit{posterior contraction rates}, have seen extremely vast developments since seminal work in the 2000s \cite{GGvdV00,SW01,GvdV07,GN11,vdVvZ08}, leading to an extensive literature that encompasses a large number of prior distributions and statistical models; see the monograph \cite{GvdV17}. For Besov priors, the first general results were derived in a recent paper by Agapiou et al.~\cite{ADH21}, based on a study of the concentration properties of product measures with heavier-than-Gaussian tails (extending the results for Gaussian priors of \cite{vdVvZ08}). They showed that, in the Gaussian white noise model, suitably `rescaled' and `under-smoothing' Besov prior contracts at optimal rates towards spatially inhomogeneous ground truths in the Besov scale. In \cite{AW23}, this result is extended to nonlinear inverse problems. A further related reference is \cite{GR22}, where optimal rates for Sobolev-regular ground truths are obtained in a reversibile multi-dimensional diffusion model. We also refer to \cite{CN13,AGR13,R13} for some earlier results, and to the preprint \cite{RR21} for results on spike-and-slab and tree-type
priors in nonparametric regression under a spatially varying Hölder smoothness assumption. Notably, alongside the white noise setting, \cite{ADH21} also considered density estimation. In this case, however, the more involved concentration properties of Besov priors compared to Gaussian priors leads to some intricate complexity bounds, ultimately resulting in their Theorem 6.7 in sub-optimal posterior contraction rates. More discussion can be found in Section \ref{Subsec:NonAdaptRates}.

	The aim of the present paper is to contribute to this recent line of work. We focus on the density estimation setting, motivated by the investigation of \cite{ADH21}. Our first main result (Theorem \ref{Theo:ContrRateFixed}) shows that, for spatially inhomogeneous true densities belonging to the $B^s_{11}$-Besov scale, carefully tuned Laplace priors do attain optimal rates. We prove this via the general contraction rate theory for independent and identically distributed (i.i.d.) observations of Ghosal et al.~\cite{GGvdV00}, employing rescaled and under-smoothing Laplace priors similar to those used in \cite[Section 5]{ADH21} and in \cite{AW23,GR22}, which we show in the proof to yield tighter complexity bounds. Theorem \ref{Theo:ContrRateFixed} thus extends the white noise result of \cite{ADH21} to the density estimation setting, reconciling the theory for Besov priors in the two statistical models.

	An interesting aspect of the above result is the combination in the design of the prior of rescaling and under-smoothing, whose interplay allows to balance in the proofs the variance component to the bias relative to spatially inhomogeneous ground truths; see the  discussion after Theorem \ref{Theo:ContrRateFixed}. As observed in \cite{ADH21}, this separates the existing theory for Besov priors and ground truths in Besov scales to the one for Gaussian priors, where under traditional regularity assumptions optimal  rates are typically obtained with priors of matching smoothness, e.g., \cite{vdVvZ08,GN16,GvdV17}. While the necessity of such prior construction remains unclear due to the absence of contraction rates lower bounds for Besov priors (cf.~Remark \ref{Rem:Necessity}), we further investigate these issues exploring the role of rescaling and under-smoothing in Sections \ref{Subsubsec:Partial} and \ref{Subsubsec:Homog}. In the former, we show that, for true densities in the $B^s_{11}$-Besov scale, the same optimal rates of Theorem \ref{Theo:ContrRateFixed} are attained (up to a logarithmic factor), by \textit{partially} rescaled under-smoothing Laplace priors. These are obtained by rescaling only the wavelets at resolutions larger than a prescribed threshold, providing an (arguably more natural) alternative prior construction that is not constrained to asymptotically shrink uniformly towards zero. In Section \ref{Subsubsec:Homog}, we instead consider the recovery of spatially \textit{homogeneous} ground truths, and show that in this case rescaling and under-smoothing are unnecessary, as (non-rescaled) Laplace priors with matching regularity attain optimal posterior contraction rates.

	For all the results described above, the specification of the prior requires knowledge of the smoothness of the unknown true density, which is often an unrealistic assumption. To overcome such limitation, we investigate \textit{adaptation} to the unknown smoothness in Section \ref{Subsec:AdaptRates}. Following the well-established hierarchical Bayesian approach (e.g., \cite[Chapter 10]{GvdV17}), we consider hierarchical, conditionally Laplace priors, obtained by randomising the prior regularity parameter via a further (hyper-)prior. Theorem \ref{Theo:ContrRateAdapt} shows that, for a carefully constructed hyper-prior (of the form considered in, e.g., \cite{LvdV07,GLvdV08,vWvZ16}), optimal posterior contraction rates are obtained for true densities in the $B^s_{11}$-Besov scale, simultaneously for any smoothness level. To the best of our knowledge, this represents the first study of adaptation for Besov priors, and also the first instance in the literature of a prior achieving adaptive rates over the $B^s_{11}$-Besov scale. Finally, the adaptive result in Theorem \ref{Theo:ContrRateAdapt} is complemented in Theorem \ref{Theo:ContrRateAdaptHomog}, where a hierarchical non-rescaled Laplace prior is shown to achieve adaptive posterior contraction rates for spatially homogeneous ground truths. After the completion of this manuscript, we learned about the independent work by Agapiou and Savva \cite{AS22}, exploring adaptation for Besov priors in the white noise model.

	We conclude the main body of the paper in Section \ref{Sec:Discussion}, where a summary and further discussion of the presented results can be found, alongside an outline of potential directions for future research. In Section \ref{Subsec:Computation}, we discuss implementation of posterior inference with Laplace priors in the present density estimation setting. All the proofs of the main results are developed in Section \ref{Sec:Proofs}. Finally, Appendix \ref{Sec:AdditionalMaterial} contains some background material on Besov priors, as well as two auxiliary results used in the proofs.

%
%
%
%
%

\section{Bayesian density estimation with Laplace priors}
\label{Sec:ProbFormul}

%
%
%
%
%

\subsection{Function spaces and other preliminaries}
\label{Subsec:Preliminaries}

Throughout, the domain of interest is the $d$-dimensional unit cube $\cube, \ d\in\N$. We denote the usual Lebesgue spaces on $[0,1]^d$ by $L^p(\cube),\ p\ge 1$, equipped with norm $\|\cdot\|_p$, and denote by $\langle\cdot,\cdot\rangle_2$ the inner product on $L^2(\cube)$. We write $C(\cube)$ for the space of continuous functions on $\cube$, equipped with the supremum-norm $\|\cdot\|_\infty$.

%

	Let $\{\psi_{lr}, \ l\in\N, \ r=1,\dots,2^{ld}\}$ be an orthonormal tensor product wavelet basis of $L^2(\cube)$, constructed from $S$-regular, $S\in\N$, compactly supported and boundary-corrected Daubechies wavelets in $L^2([0,1])$; see, e.g., \cite[Chapter 4.3]{GN16} for details. In what follows, we tacitly assume $S$ to be sufficiently large, in particular, greater than the smoothness parameter $s\ge0$ appearing below. For a resolution level $L\in\N$, define the finite-dimensional approximation space
$$
	V_L := \textnormal{span}\{\psi_{lr}, \ l=1,\dots,L, \ r=1,\dots,2^{ld}\}
$$
and let $P_L:L^2(\cube)\to V_L$ be the associated projection operator. Note that $V_L$ has dimension $\textnormal{dim}(V_L)=O(2^{Ld})$ as $L\to\infty$. For a smoothness parameter $s\ge 0$ and integrability indices $p,q\in[1,\infty]$, define the Besov space $B^s_{pq}(\cube)$ via its wavelet characterisation (cf.~\cite{GN16}, p.370f):
$$
	B^s_{pq}(\cube) := \left\{w \in L^p(\cube)
	 :  \|w\|_{B^s_{pq}}^q := \sum_{l=1}^\infty2^{ql\left(s - \frac{d}{p} + \frac{d}{2} \right)}
	\Bigg(\sum_{r=1}^{2^{ld}} |\langle w,\psi_{lr}\rangle_2|^p \Bigg)^\frac{q}{p}<\infty\right\},
$$
where the above $\ell_p$- and $\ell_q$-sequence space norms are replaced by the $\ell_\infty$-norm if $p=\infty$ or $q=\infty$ respectively. Recall that the Besov scale contains the traditional Sobolev spaces $H^s(\cube)$ and H\"older spaces $C^s(\cube)$: $B^s_{22}(\cube) = H^s(\cube)$ and $C^s(\cube)\subseteq B^s_{\infty\infty}(\cube)$ for all $s>0$ (with equality holding when $s\notin\N$). As mentioned in the introduction, for $p=q=1$, the spaces $B^s_{11}(\cube)$ instead provide a mathematical model for spatially inhomogeneous functions. See \cite[Section 1]{DJ98}, or also \cite{GN16,KLSS23}, for the relationship between the spaces $B^1_{1q}(\cube)$ and the space of bounded variation functions.

	When no confusion may arise, we suppress the dependence of the function spaces on the underlying domain, writing for example $B^s_{pq}$ instead of $B^s_{pq}(\cube)$. We denote by $\lesssim,\ \gtrsim$, and $\simeq$, respectively one- or two-sided inequalities holding up to multiplicative constants. We write $N(\xi;\Pcal,d), \ \xi>0$, for the $\xi$-covering number of a set $\Pcal$ with respect to a semi-metric $d$ on $\Pcal$, defined as the minimal number of balls of radius $\xi$ in the metric $d$ needed to cover $\Pcal$. Positive numerical constants in the proofs are denoted by $c_1,c_2,\dots>0$.

%
%
%
%
%

\subsection{Observation model, prior and posterior}
\label{Subsec:ObsAndPrior}

Consider a sample $X^{(n)}:=(X_1,\dots,X_n)$ of i.i.d.~random variables with values in $\cube, \ d\in\N,$ from an unknown probability distribution $P_0$ with density function $p_0$ relative to the $d$-dimensional Lebesgue measure $dx$. This paper is concerned with the problem of estimating $p_0$ from the random sample $X^{(n)}$. The main focus is on the case where $p_0$ is (possibly) spatially inhomogeneous, e.g., possibly flat in some parts of the domain and spiky in others. A natural mathematical model for such setting is then to assume a (Borel measurable) parameter space $p_0\in\Pcal\subseteq B^s_{11}(\cube)$ for some $s\ge0$; see Section \ref{Subsec:Preliminaries} for definitions and details.

	We adopt a Bayesian approach, regarding the data as conditionally independent given a density $p$, $X_1,\dots,X_n|p\iid p$, and assigning to $p$ a (possibly $n$-dependent) prior distribution $\Pi_n$ supported on the parameter space $\Pcal$. This induces the posterior distribution $p|X^{(n)}\sim \Pi_n(\cdot|X^{(n)})$, which by Bayes' formula (e.g., \cite{GvdV17}, p.7) equals
$$
	\Pi_n(A|X^{(n)}) 
	= \frac{\int_A \prod_{i=1}^np(X_i)d\Pi_n(p)}{\int_\Pcal \prod_{i=1}^np'(X_i)d\Pi_n(p')},
	\qquad A\subseteq \Pcal\ \textnormal{measurable}.
$$
In the following, we study the (frequentist) consistency of $\Pi_n(\cdot|X^{(n)})$, assuming that the observations $X_1\dots,X_n$ are drawn from some fixed `true' unknown density $p_0\in\Pcal$ to be estimated, and investigating the asymptotic concentration of the posterior distribution around $p_0$ when $n\to\infty$.

	The recovery performance of the posterior distribution is known to depend on the choice of the prior, which is a key modelling step in the Bayesian approach. To reflect the potential spatial inhomogeneity of $p_0\in B^s_{11}(\cube)$, we employ (Besov-)Laplace priors. These represent a particular instance in the class of Besov priors introduced by Lassas et al.~\cite{LSS09}, defined in general as random wavelet expansions with i.i.d.~random coefficients following the probability density function $\lambda_p(z)\propto e^{-|z|^p/p}, \ z\in\R, \ p\in[1,2]$. Laplace priors arise in the case of Laplace-distributed coefficients ($p=1$), and were recently shown by Agapiou et al.~\cite{ADH21} to yield, in the white noise model, optimal estimation of spatially inhomogeneous functions in Besov scales. In accordance with the latter references, we define Laplace priors for densities starting from  
a random function
\begin{equation}
\label{Eq:Wn}
	W_n(x) 
	= \sum_{l=1}^\infty \sum_{r=1}^{2^{ld}} \sigma_{n,lr}W_{lr}\psi_{lr}(x), \qquad x\in [0,1]^d,
\end{equation}
where $\{\psi_{lr}, \ l\in\N, \ r=1,\dots,2^{ld}\}$ is an orthonormal tensor product wavelet basis of $L^2(\cube)$ as described in Section \ref{Subsec:Preliminaries}, $\sigma_{n,lr}>0$ are deterministic (possibly $n$-dependent) scaling factors satisfying 
\begin{equation}
\label{Eq:Scaling}
	\sigma_{n,lr} = O(2^{-l(t+d/2)}), 
	\qquad \textnormal{as} \ l\to\infty,
	\qquad \textnormal{some}\ t>0,
\end{equation}
and $W_{lr}$ are i.i.d.~random coefficients following the Laplace distribution, whose density equals
$$
	\lambda_1(z) = \frac{1}{2} e^{-|z|}, \qquad z\in\R.
$$
The decay of the scaling factors determines the regularity properties of realisations of the random function $W_n$: in particular, a simple computation shows that condition \eqref{Eq:Scaling} implies that $W_n\in C(\cube)\cap B^{t'}_{rr}(\cube)$ almost surely for all $t'<t$ and $r\in[1,\infty]$ (cf.~Lemma 5.2 and Proposition 6.1  in \cite{ADH21}, or also Lemma 7.1 in \cite{AW23}). Under this condition, we may regard $W_n$ as a Borel measurable random element in the separable Banach space $C(\cube)$, whose law we denote by $\Pi_{W_n}$. Following the terminology of \cite{ADH21,AS22}, we call $\Pi_{W_n}$ a \textit{$t$-regular Laplace prior} on $C(\cube)$. Note the slightly different parametrisation compared to \cite{LSS09}, where the law of $W_n$ in \eqref{Eq:Wn} is referred to as a $B^{t+d}_{11}$-Besov prior.

	Given $W_n$ as in \eqref{Eq:Wn}, we then construct a prior on the set of densities functions on $\cube$ by taking the law $\Pi_n$ of the random function
\begin{equation}
\label{Eq:PhiWn}
	\phi_{W_n}(x) := \frac{\phi(W_n(x))}{\int_{\cube}\phi(W_n(x'))dx'},
	\qquad x\in\cube,
\end{equation}
where $\phi:\R\to[0,\infty)$ is a positive, strictly increasing and smooth link function. A common choice is the exponential link function $\phi(z) = e^z, \ z\in\R$, but for some of the results to follow it will prove useful to employ link functions affording better control over the tails. In a slight abuse of terminology, we call $\Pi_n$ a $t$-regular Laplace prior on densities. Posterior computation with such priors in the density estimation model is discussed in Section \ref{Subsec:Computation}.

%
%
%
%
%

\section{Main results}
\label{Sec:MainResults}

In this section we present our main results concerning the asymptotic behaviour of posterior distributions $\Pi_n(\cdot|X^{(n)})$ arising from the Laplace priors  on densities $\Pi_n$ introduced in Section \ref{Subsec:ObsAndPrior}. We quantify the speed at which $\Pi_n(\cdot|X^{(n)})$ concentrates around the true density $p_0$ generating the data according to the usual notion of \textit{posterior contraction rates}, that is sequences $\xi_n\to0$ such that, for large enough $M>0$,
$$
	\Pi_n\left( p :  d(p,p_0)> M\xi_n \Big| X^{(n)}\right) \to 0
$$
in $P_0$-probability as $n\to\infty$. Above, $d$ is a distance between probability densities. In the present paper, we will mostly deal with the \textit{total variation distance}
$$
	d_{TV}(p,p') = \frac{1}{2}\|p - p'\|_1,
$$
which, due to its characterisation as (half) $L^1$-norm, is naturally aligned with the $L^1$-structure underlying the Besov spaces $B^s_{11}(\cube)$.

%
%
%
%
%

\subsection{Posterior contraction rates for Laplace priors with fixed regularity}

%
%
%
%
%

\subsubsection{Rescaled under-smoothing Laplace priors}
\label{Subsec:NonAdaptRates}

We first consider the case where the scaling factors $\sigma_{n,lr}$ in \eqref{Eq:Wn} are chosen as deterministic sequences. Specifically, for fixed $s>d$, we take $\sigma_{n,lr} = 2^{-l(s-d/2)}n^{-d/(2s + d)}$, resulting in a rescaled $(s-d)$-regular Laplace prior $\Pi_{W_n}$ arising as the law of
\begin{equation}
\label{Eq:WnFixed}
	W_n 
	= 
	\frac{1}{n^\frac{d}{2s + d}}
	\sum_{l=1}^\infty \sum_{r=1}^{2^{ld}} 2^{-l\left(s - \frac{d}{2}\right)}
	W_{lr}\psi_{lr},
	\qquad W_{lr}\iid \textnormal{Laplace}.
\end{equation}
We then construct a prior $\Pi_n$ on densities as in \eqref{Eq:PhiWn}. We here allow for any log-Lipschitz link function (e.g., the commonly used exponential link function, but also any of the more restrictive ones appearing in the results below). The next theorem shows that the posterior distribution resulting from $\Pi_n$ contracts around the true density $p_0\in B^s_{11}(\cube)$ in total variation distance at optimal rate.

\begin{theorem}\label{Theo:ContrRateFixed}
For fixed $s>d$, let the prior $\Pi_n$ be constructed as in \eqref{Eq:PhiWn} for $W_n$ as in \eqref{Eq:WnFixed} and $\phi:\R\to(0,\infty)$ a strictly increasing, bijective and smooth function with uniformly Lipschitz logarithm. Let $X_1,\dots,X_n$ be a random sample from a probability density $p_0\in B^s_{11}(\cube)$ satisfying $p_0(x)>0$ for all $x\in\cube$.  Then, for $M>0$ large enough,
$$
	\Pi_n\left( p  : \ d_{TV}(p,p_0)> M n^{-\frac{s}{2s + d}} \Big| X^{(n)}\right) \to 0
$$
in $P_0$-probability as $n\to\infty$.
\end{theorem}

%


%

	The obtained posterior contraction rate $n^{-s/(2s+d)}$ is known to be the minimax optimal rate for estimating $p_0\in B^s_{11}(\cube)$ from the random sample $X^{(n)}$ in $L^p$-loss for any $p<2s + d$; see \cite[Theorem 10.3]{HKPT98} (whose proof techniques naturally extend to the multi-dimensional case $d\ge2$). Due to the identity $d_{TV}(p,p') = \frac{1}{2}\|p - p'\|_1$, the rate is optimal for the total variation distance. Furthermore, an inspection of the proof shows that the claim of Theorem \ref{Theo:ContrRateFixed} remains valid also if the total variation distance is replaced by either the \textit{Hellinger distance} (cf.~\eqref{Eq:HellDist}) or the $L^2$-distance; see the discussion preceding Lemma \ref{Lem:SievesFixed} below for details.

	Theorem \ref{Theo:ContrRateFixed} improves upon the density estimation results in \cite[Section 6]{ADH21}, which only considers the case of \textit{spatially homogeneous densities} $p_0\in B^s_{\infty\infty}([0,1])$ and where only \textit{polynomially suboptimal} posterior contraction rates are obtained. Here we show that, even for {(possibly)} spatially inhomogeneous $p_0\in B^s_{11}(\cube)$, suitably tuned Besov priors achieve optimal rates. The proof follows a similar strategy to the results in \cite{ADH21}, based on the general posterior contraction rate theory of Ghosal et al.~\cite{GGvdV00}, but crucially uses a different tuning of the prior which combines an under-smoothing effect with rescaling by the  factor $n^{-d/(2s+d)}$. In particular, the rescaling implies that the prior $\Pi_n$ asymptotically concentrates over \textit{sieve sets} of densities that are uniformly bounded and bounded away from zero (cf.~Lemma \ref{Lem:SievesFixed}), for which, compared to \cite{ADH21}, tighter complexity bounds with respect to the relevant distances (respectively, the total variation or Hellinger distances) can be obtained.

	Expanding on the tuning of the prior in \eqref{Eq:WnFixed}, draws from $\Pi_n$ almost surely lie in $B^t_{11}(\cube)$ only for $t<s-d$; see Section \ref{Subsec:PropBesovPriors}. Since in Theorem \ref{Theo:ContrRateFixed} the true density $p_0$ is assumed to be in $B^s_{11}(\cube)$, the prior is seen to be \textit{under-smoothing}. This contrasts with a number of results in the literature on Gaussian priors, where, under Sobolev- or H\"older-type regularity assumptions on the ground truth, optimal contraction rates are obtained with priors matching the true smoothness. See Section 3.1 in van der Vaart and van Zanten \cite{vdVvZ08} for results in density estimation, and \cite[Chapter 7.3]{GN16} or \cite[Chapter 11]{GvdV17} for a general overview on the theory. In the case of (possibly) spatially inhomogeneous $p_0\in B^s_{11}(\cube)$ and {Laplace} priors, the {effect of the interplay between} under-smoothing and rescaling can be explained in terms of the \textit{bias-variance tradeoff}: as shown in the calculations of \cite{ADH21}, it turns out that for {Laplace} priors with matching regularity the bias term appearing in the proof is too large due to the misalignment between the $L^1$-structure underlying the spaces $B^s_{11}(\cube)$ and the \textit{Kullback-Leibler divergence} and \textit{variation}  in the \textit{prior mass condition} \eqref{Eq:KLCond} below. Under-smoothing allows to reduce the bias (in particular, sidestepping the necessity of approximating $p_0$), while the corresponding variance increase is balanced via the rescaling by the decaying factor $n^{-d/(2s+d)}$. This joint effect was already observed in \cite[Section 5.2]{ADH21}, where in the white noise model optimal contraction rates towards spatially inhomogeneous signals are obtained using {rescaled under-smoothing} Besov priors with a similar tuning to \eqref{Eq:WnFixed}. See the recent work by Agapiou and Wang \cite{AW23} for an extension to nonlinear inverse problems. In fact, under traditional regularity assumptions on the ground truth, under-smoothing rescaled \textit{Gaussian} priors have been successfully employed in the context of nonlinear inverse problems, e.g., \cite{MNP21a,AN19,GN20,NW22, K22}, and reversible diffusion models \cite{GR22}.

\begin{remark}[On the necessity of under-smoothing and rescaling]\label{Rem:Necessity}
	While the combination of under-smoothing and rescaling plays a crucial role in our proof, its necessity in the presence of spatially inhomogeneous ground truths is an interesting open question in the theory of Besov priors. In particular, as the lower bounds techniques for the contraction rates of Gaussian priors developed by Castillo \cite{Cast08} do not extend to Besov priors (cf.~Remark 3.3 in \cite{ADH21}), it remains unclear wether the observed sub-optimality of non-rescaled Besov priors with matching regularity is a fundamental phenomenon or an artefact of the existing proofs. In Sections \ref{Subsubsec:Partial} and \ref{Subsubsec:Homog}, we further consider this matter investigating respectively the performance of partially rescaled Laplace priors under true densities in the $B^s_{11}$-Besov scale, as well as non-rescaled Laplace priors in the spatially homogeneous setting.
\end{remark}

\begin{remark}[Extensions to general Besov priors]\label{Rem:Extensions}
The results of \cite{ADH21} in the white noise model suggest that the investigation carried out in the present paper could be extended to general Besov priors and to true densities $p_0$ in general Besov scales $B^s_{qq}(\cube)$, $s\ge 0$, $q\ge1$, at the expense of slightly more technical proofs. We refrained to pursue such extensions here, focusing our results on the primarily relevant case of Laplace priors and $p_0\in B^s_{11}(\cube)$.
\end{remark}
	
%
%
%
%
%

\subsubsection{Partially-rescaled Laplace priors}
\label{Subsubsec:Partial}

As observed in the discussion following Theorem \ref{Theo:ContrRateFixed}, the rescaling factor $n^{-d/(2s+d)}$ in \eqref{Eq:WnFixed} balances the increased variance resulting from under-smoothing. Furthermore, it also asymptotically implies a bound on the norm of prior draws, which yields in the proof tight complexity bounds for the associated sieve sets. On the other hand, from a Bayesian perspective, the rescaling is arguably not satisfactory as {it causes} the prior to shrink towards zero: in particular, $\|W_n\|_\infty\to0$ almost surely as $n\to\infty$. A more refined investigation of the prior geometry  (cf.~the proof of Theorem \ref{Theo:ContrRatePartial} in Section \ref{Subsec:ProofTheoPartial} below) however reveals that the aforementioned variance increase is - in some sense -  only effective at resolution levels $l$ such that $2^{ld}$ is larger than the usual optimal dimension $n^{d/(2s + d)}$ for estimating an $s$-regular function. This motivates the construction of the following partially rescaled $(s-d)$-regular {Laplace} prior:
\begin{equation}
\label{Eq:WnFixedPartial}
	W_n 
	= 
	\sum_{l=1}^{L_n}\sum_{r=1}^{2^{ld}} 2^{-l\left(s - \frac{d}{2}\right)}
	W_{lr}\psi_{lr}
	+
	\frac{1}{n^\frac{d}{2s + d}\log n}
	\sum_{l=L_n+1}^\infty \sum_{r=1}^{2^{ld}} 2^{-l\left(s - \frac{d}{2}\right)}
	W_{lr}\psi_{lr},
	\ \ W_{lr}\iid \textnormal{Laplace},
\end{equation}
where $L_n\in \N$ is chosen so that $2^{L_n}\simeq n^{1/(2s+d)}$. A prior $\Pi_n$ on densities is then constructed as previously following \eqref{Eq:PhiWn}. Here, to deal with the weaker control over the norm of $\phi_{W_n}$ resulting from the partial rescaling, we employ uniformly Lipschitz link functions that are bounded away from zero. This allows in the proof to construct sieve sets with sufficiently small complexity. The next theorem shows that the resulting posterior distribution attains (up to a log-factor) the same optimal rate of Theorem \ref{Theo:ContrRateFixed}.


\begin{theorem}\label{Theo:ContrRatePartial}
For fixed $B>0$ and $s>d$, let the prior $\Pi_n$ be constructed as in \eqref{Eq:PhiWn} for $W_n$ as in \eqref{Eq:WnFixedPartial} and $\phi:\R\to(B,\infty)$ a strictly increasing, bijective, smooth and uniformly Lipschitz function. Let $X_1,\dots,X_n$ be a random sample from a probability density $p_0 = f_0/\int_{\cube}f_0(x)dx$ for some $f_0\in B^s_{11}(\cube)$  satisfying $f_0(x) > B$ for all $x\in \cube$. Then, for $M>0$ large enough,
$$
	\Pi_n\left( p  : \ d_{TV}(p,p_0)> M n^{-\frac{s}{2s + d}}\sqrt{\log n} 
	\Big| X^{(n)}\right) \to 0
$$
in $P_0$-probability as $n\to\infty$.
\end{theorem}

	Instances of link functions satisfying the assumptions of Theorem \ref{Theo:ContrRatePartial} are certain \textit{regular} link functions appearing in the Bayesian inverse problems literature \cite{AN19,NvdGW20,GN20,K22}. As a concrete example, take
$$
	\phi(z) = B + \frac{1- B}{g*\eta(0)}g*\eta(z), \qquad \eta(z) 
	= e^z1_{\{z<0\}} + (1+z)1_{\{z\ge 0\}},
$$
where $g:\R\to [0,\infty)$ is a smooth and compactly supported function with $\int_\R g(z)dz =1$ (cf.~Example 8 in \cite{NvdGW20}).

\begin{remark}[Link function lower bound]\label{Rem:LB}
	The conclusion of Theorem \ref{Theo:ContrRatePartial} is restricted to densities that are bounded away from zero, with `core' $f_0$ point-wise greater than the constant $B$ lower bounding the range of the link function $\phi$. The latter is a relatively mild assumption as any arbitrary small but fixed $B>0$ is allowed, and choosing smaller values of $B$ only affects the constant premultiplying the obtained contraction rate. In fact, a  direct adaptation of the proof of Theorem \ref{Theo:ContrRatePartial} implies that taking an $n$-dependent lower bound $B=B_n \simeq 1/\log n$ leads to posterior contraction towards any density $p_0\in B^s_{11}(\cube)$ bounded away from zero at a rate that deteriorates only by an additional log-factor $(\log n)^c$ for some $c>0$.
\end{remark}
%
%
%
%
%

\subsubsection{Posterior contraction rates for spatially homogeneous densities}
\label{Subsubsec:Homog}

We conclude the investigation on {Laplace} priors with fixed regularity considering the case of spatially homogeneous densities $p_0\in B^s_{\infty\infty}(\cube)$, which is the setting studied in Section 6 of \cite{ADH21}. Under this assumption, the Kullback-Leibler divergence and variation are well-aligned to the regularity structure of the ground truth, so that, as opposed to the inhomogeneous case considered previously (cf.~the discussion after Theorem \ref{Theo:ContrRateFixed}), non-rescaled {Laplace} priors with matching regularity attain the necessary balance of bias and variance. This paves the way to obtaining optimal posterior contraction rates, as we illustrate in the next theorem considering the following (truncated) $s$-regular {Laplace} prior
\begin{equation}
\label{Eq:WnFixedTrunc}
	W_n 
	= 
	\sum_{l=1}^{L_n}\sum_{r=1}^{2^{ld}} 2^{-l\left(s + \frac{d}{2}\right)}
	W_{lr}\psi_{lr},
	\qquad W_{lr}\iid \textnormal{Laplace},
\end{equation}
where $L_n\in \N$ is such that $2^{L_n}\simeq n^{1/(2s+d)}$.

\begin{theorem}\label{Theo:ContrRateTrunc}
For fixed $s>d$, let the prior $\Pi_n$ be constructed as in \eqref{Eq:PhiWn} for $W_n$ as in \eqref{Eq:WnFixedTrunc} and $\phi:\R\to(0,\infty)$ a strictly increasing, bijective and smooth function with uniformly Lipschitz logarithm. Let $X_1,\dots,X_n$ be a random sample from a probability density $p_0\in B^s_{\infty\infty}(\cube)$ satisfying $p_0(x)>0$ for all $x\in\cube$.  Then, for $M>0$ large enough,
$$
	\Pi_n\left( p  : \ d_{TV}(p,p_0)> M n^{-\frac{s}{2s + d}} \Big| X^{(n)}\right) \to 0
$$
in $P_0$-probability as $n\to\infty$.
\end{theorem}

Truncating the prior at level $L_n$ is here natural since, as remarked before Theorem \ref{Theo:ContrRatePartial}, $2^{L_nd}\simeq n^{d/(2s + d)}$ is the optimal dimension for estimating a (spatially homogeneous) $s$-regular function. In the proof, the truncation reduces the complexity of the sieve sets associated to $W_n$, allowing the derivation of suitable bounds on their metric entropy.  For ground truths with Sobolev regularity, similar truncated Besov priors were recently shown to yield optimal contraction rates in a reversible diffusion model \cite{GR22}.

%
%
%
%
%

\subsection{Adaptive posterior contraction rates}
\label{Subsec:AdaptRates}

\subsubsection{Hierarchical rescaled Laplace Priors}
\label{Subsec:AdaptRatesBesov}

The optimal contraction rates obtained in Theorems \ref{Theo:ContrRateFixed}, \ref{Theo:ContrRatePartial} (up to a log-factor) and \ref{Theo:ContrRateTrunc} depend in an essential way on the appropriate specification of the prior regularity relative to the smoothness of the ground truth: for $p_0\in B^s_{11}(\cube)$, Theorems \ref{Theo:ContrRateFixed} and \ref{Theo:ContrRatePartial} employ $(s-d)$-regular {Laplace} priors, while for $p_0\in B^s_{\infty\infty}(\cube)$, Theorem \ref{Theo:ContrRateTrunc} assumes prior regularity $s$. In each case, choosing different smoothness parameters (and/or different rescalings in Theorems \ref{Theo:ContrRateFixed} and \ref{Theo:ContrRatePartial})  would yield through our proofs sub-optimal rates, in accordance to the results of Agapiou et al.~\cite{ADH21}, where rates for various prior regularities are obtained, and to related literature on Gaussian priors, e.g., \cite{vdVvZ08,C08,KvdVvZ11}. The results presented in the previous sections are thus \textit{non-adaptive}, in that the construction of the prior requires knowledge of the smoothness of the true density. As this is often an unrealistic assumption, it is of interest to construct a Bayesian procedure based on {Laplace} priors that, not requiring knowledge of the regularity of $p_0$, automatically adapts to it attaining optimal contraction rates.

	An established method to achieve adaptation in Bayesian procedures is by randomising the prior regularity parameter, assigning to it a further (hyper-)prior; see, e.g., \cite[Chapter 10]{GvdV17}. Here, we pursue this approach. In particular, we employ a hierarchical, conditionally {Laplace} prior $\Pi_{W_n}$ arising as the law of
\begin{align}
\label{Eq:WnHierarc}
	W_n &=
	\frac{1}{n^\frac{d}{2S + d}}
	\sum_{l=1}^\infty \sum_{r=1}^{2^{ld}} 2^{-l\left(S - \frac{d}{2}\right)}
	W_{lr}\psi_{lr},
	\qquad W_{lr}\iid \textnormal{Laplace},
	\qquad S \sim \Sigma_n,
\end{align}
where $\Sigma_n$ is an absolutely continuous  ($n$-dependent) distribution on $(0,\infty)$ with density $\sigma_n$. We take a specific choice for $\Sigma_n$, assuming it to be supported on the increasing interval $(d,\log n]$, with density
\begin{equation}
\label{Eq:HyperPrior}
	\sigma_n(s) = \frac{e^{-n^{d/(2s+d)}}}{\zeta_n}, \qquad s\in (d,\log n],
\end{equation}
where $\zeta_n\simeq \log n$ is the normalising constant. Conditionally given $S$, $W_n|S$ thus corresponds to the rescaled under-smoothing {Laplace} prior considered in Section \ref{Subsec:NonAdaptRates}. For any $S>d$, $W_n\in C(\cube)$ with conditional probability given $S$ equal to one (cf.~Section \ref{Subsec:PropBesovPriors}), implying that the hierarchical prior $\Pi_{W_n}$ is supported on $C(\cube)$.

	Given $W_n$ as in \eqref{Eq:WnHierarc}, a hierarchical {Laplace} prior on densities $\Pi_n$ is constructed as in \eqref{Eq:PhiWn}, via a link function $\phi$ that, as in Section \ref{Subsubsec:Partial}, we require to be uniformly Lipschitz and bounded away from zero. The next theorem shows that the associated posterior distribution attains optimal posterior contraction rates towards densities $p_0\in B^s_{11}(\cube)$ of any smoothness $s_0\in(d,\infty)$.

\begin{theorem}\label{Theo:ContrRateAdapt}
For fixed $B>0$, let the prior $\Pi_n$ be constructed as in \eqref{Eq:PhiWn} for $W_n$ as in \eqref{Eq:WnHierarc} and $\phi:\R\to(B,\infty)$ a strictly increasing, bijective, smooth and uniformly Lipschitz function. Let $X_1,\dots,X_n$ be a random sample from a probability density $p_0 = f_0/\int_{\cube}f_0(x)dx$ for some $f_0\in B^{s_0}_{11}(\cube)$, any $s_0>d$, satisfying $f_0(x) > B$ for all $x\in \cube$. Then, for $M>0$ large enough,
$$
	\Pi_n\left( p  : \ d_{TV}(p,p_0)> M n^{-\frac{s_0}{2s_0 + d}} 
	\Big| X^{(n)}\right) \to 0
$$
in $P_0$-probability as $n\to\infty$.
\end{theorem}

	Similar to Theorem \ref{Theo:ContrRatePartial}, the stronger requirements on the link function $\phi$ play a crucial role in the proof to handle the weaker control, resulting from the hierarchical construction, over the norm of $W_n$ in \eqref{Eq:WnHierarc} compared to the rescaled {Laplace} priors of fixed regularity of Theorem \ref{Theo:ContrRateFixed}. We remark that the above lower bound $B$ can be chosen arbitrarily small, and in fact that taking $B=B_n\simeq 1/ \log n$ would allow to extend Theorem \ref{Theo:ContrRateAdapt} to any density $p_0\in B^s_{11}(\cube)$ bounded away from zero, with the rate $n^{-s_0/(2s_0 + d)}$ replaced by $n^{-s_0/(2s_0 + d)}(\log n)^c$ for some $c>0$.

	Regarding the specific choice of the hyper-prior, note that $\sigma_n(s)$ is proportional to $e^{-n\varepsilon_{s,n}^2}$, where $\varepsilon_{s,n} := n^{-s/(2s+d)}$ is the optimal rate for estimating $p_0\in B^s_{11}(\cube)$ obtained in Theorem \ref{Theo:ContrRateFixed} using a rescaled $(s-d)$-regular {Laplace} prior. This choice is motivated by previous findings in the literature that showed that, under some generality, hyper-priors of this kind can lead to adaptation in various statistical models, including in density estimation \cite{LvdV07,GLvdV08} and drift estimation for diffusion processes \cite{vWvZ16}. In accordance to the latter references, Theorem \ref{Theo:ContrRateAdapt} shows that this is indeed the case in the setting of {Laplace} priors and spatially inhomogeneous densities $p_0\in B^s_{11}(\cube)$.

\subsubsection{Adaptive rates for spatially homogeneous densities}
\label{Subsec:AdaptRatesHomog}

	We complement the adaptive result in Theorem \ref{Theo:ContrRateAdapt} by showing that Besov priors can achieve adaptation also over the $B^s_{\infty\infty}$-Besov scale of spatially homogeneous functions. Motivated by the findings of Sections \ref{Subsubsec:Homog} and \ref{Subsec:AdaptRatesBesov}, we consider hierarchical non-rescaled (truncated) {Laplace} priors $\Pi_{W_n}$ arising as the law of
\begin{equation}
\label{Eq:HierTrunc}
	W_n 
	= 
	\sum_{l=1}^{L_{S,n}}\sum_{r=1}^{2^{ld}} 2^{-l\left(S + \frac{d}{2}\right)}
	W_{lr}\psi_{lr},
	\qquad W_{lr}\iid \textnormal{Laplace},
	\qquad 2^{L_{S,n}}\simeq n^{1/(2S+d)},
	\qquad S\sim \Sigma_n,
\end{equation}
where $\Sigma_n$ is the hyper-prior supported on $(d,\log n]$ with density $\sigma_n$ defined in \eqref{Eq:HyperPrior}. Note that, conditionally given $S$, $W_n|S$ here corresponds to the non-rescaled truncated Laplace prior employed in Section \ref{Subsubsec:Homog}.

\begin{theorem}\label{Theo:ContrRateAdaptHomog}
Let the prior $\Pi_n$ be constructed as in \eqref{Eq:PhiWn} for $W_n$ as in \eqref{Eq:HierTrunc} and $\phi:\R\to(0,\infty)$ a strictly increasing, bijective and smooth function with uniformly Lipschitz logarithm. Let $X_1,\dots,X_n$ be a random sample from a probability density $p_0 \in B^{s_0}_{\infty\infty}(\cube)$, any $s_0>d$, satisfying $p_0(x) > 0$ for all $x\in \cube$. Then, for $M>0$ large enough,
$$
	\Pi_n\left( p  : \ d_{TV}(p,p_0)> M n^{-\frac{s_0}{2s_0 + d}} 
	\Big| X^{(n)}\right) \to 0
$$
in $P_0$-probability as $n\to\infty$.
\end{theorem}

%
%
%
%
%

\section{Summary and discussion}
\label{Sec:Discussion}

%
%
%
%
%

\subsection{Outlook}
\label{Subsec:Outlook}

In this paper, we have studied the recovery performance of Besov{-Laplace} priors in density estimation. Our main results show that, for (possibly) spatially inhomogeneous densities $p_0\in B^s_{11}(\cube)$, suitably calibrated, rescaled and under-smoothing {Laplace} priors attain optimal posterior contraction rates, and that randomising the smoothness parameter in the prior construction leads to adaptation. While we focused on density estimation, we expect our techniques to be applicable to other nonparametric statistical models such as classification and nonparametric regression. In particular, the hierarchical prior construction in Section \ref{Subsec:AdaptRatesBesov} should yield adaptive rates in the latter models under Besov regularity assumptions on the ground truth.

	The role of under-smoothing and rescaling has been discussed after Theorem \ref{Theo:ContrRateFixed}, and subsequently explored in Sections \ref{Subsubsec:Partial} and \ref{Subsubsec:Homog}. It is an interesting open question, not limited to the density estimation setting, wether non-rescaled Besov priors with matching regularity can achieve optimal rates for spatially inhomogeneous ground truths, as the observed sub-optimality \cite{ADH21} might be an artefact of the existing proofs (cf.~Remark \ref{Rem:Necessity}).

	Regarding Theorems \ref{Theo:ContrRateAdapt} and \ref{Theo:ContrRateAdaptHomog}, it is conceivable that other hyper-priors on the smoothness could lead to adaptation, including potentially hyper-priors not depending on the sample size $n$, such as the ones employed in Knapik et al.~\cite{KSvdVvZ15}. Our proof strategy, based on the general contraction rate theory of \cite{GGvdV00}, heavily relies on the specific choice \eqref{Eq:HyperPrior}, and extensions to other hyper-priors appear to require substantial modifications or different mathematical techniques. These issues represent interesting directions for future research.

\subsection{Posterior computation}
\label{Subsec:Computation}

The {Laplace} prior $\Pi_{W_n}$ introduced in Section \ref{Subsec:ObsAndPrior} naturally lends itself to discretisation by truncating the series in \eqref{Eq:Wn} at some fixed resolution level $L\in\N$. 
Identifying a function $w = \sum_{l=1}^L \sum_{r=1}^{2^{ld}}w_{lr}\psi_{lr}$ with the vector $\bold w = (w_{lr}, \ l=1,\dots,L, \ r=1,\dots, 2^{ld})\in \R^{\textnormal{dim}(V_L)}$, this yields a prior distribution $\Pi^L_{W_n}$ on $\bold w$ given by a product of Laplace distributions, with density
$$
	\pi^L_{W_n}(\bold w) 
	= \prod_{l=1}^L\prod_{r=1}^{2^{ld}}\frac{1}{2\sigma_{n,lr}}
	e^{-\frac{|w_{lr}|}{\sigma_{n,lr}}}
	=\frac{e^{-\sum_{l=1}^L\sum_{r=1}^{2^{ld}}\frac{|w_{lr}|}{\sigma_{n,lr}}}}
	{2^{\textnormal{dim}(V_L)}\prod_{l=1}^L\prod_{r=1}^{2^{ld}}\sigma_{n,lr}}
	.
$$
Then, given observations $X^{(n)} = (X_1,\dots,X_n)$, and a fixed link function $\phi:\R\to[0,\infty)$, the corresponding posterior distribution $\Pi^L_{W_n}(\cdot|X^{(n)})$ of $\bold w|X^{(n)}$ has density
\begin{align*}
	\pi^L_{W_n}(\bold w|X^{(n)}) 
	&\propto \prod_{i=1}^n \phi_{\sum_{l=1}^L \sum_{r=1}^{2^{ld}}w_{lr}\psi_{lr}}(X_i)
	\pi^L_{W_n}(\bold w)\\
	&\propto\frac{\prod_{i=1}^n \phi\left(\sum_{l=1}^L 
	\sum_{r=1}^{2^{ld}} w_{lr}\psi_{lr}(X_i) \right)}
	{\prod_{i=1}^n\int_{\cube}\phi\left(\sum_{l=1}^L 
	\sum_{r=1}^{2^{ld}} w_{lr}\psi_{lr}(x')\right)dx'}
	e^{-\sum_{l=1}^L\sum_{r=1}^{2^{ld}}\frac{|w_{lr}|}{\sigma_{n,lr}}}.
\end{align*}
For concrete choices of $\sigma_{lr}$, the right hand side can be computed from the data, numerically approximating the integrals in the denominator. While potentially computationally intensive for large dimensions $d$, this provides a starting point for implementing a Markov chain Monte Carlo (MCMC) algorithm (e.g., of Metropolis-Hastings type) to sample from $\Pi^L_{W_n}(\cdot|X^{(n)})$. Recent advances in the development of MCMC algorithms  suited to the infinite-dimensional setting based on Besov (or Besov-like) priors are in \cite{WBSCM17,CDPS19,H19}, where further references can be found. Developing efficient algorithms for posterior sampling in density estimation with Besov priors represents another interesting avenue for future study.

	In order to sample from the posterior distribution resulting from the
hierarchical priors of Section \ref{Subsec:AdaptRates}, MCMC methods for {Laplace} priors of fixed regularity (such as the one outlined above) can be employed within a Gibbs-type sampling scheme that exploits the conditionally {Laplace} structure of the prior. The algorithm would then alternate, for a given regularity $S$, an MCMC step targeting the marginal posterior distribution of $\bold w | (X^{(n)},S$), followed by, given the actual sample of $\bold w$, a second MCMC run targeting the marginal posterior distribution of $S|(X^{(n)},\bold w)$.

%
%
%
%
%

\section{Proofs}
\label{Sec:Proofs}

The proofs of Theorems \ref{Theo:ContrRateFixed}, \ref{Theo:ContrRatePartial}, \ref{Theo:ContrRateTrunc} and \ref{Theo:ContrRateAdapt} are based on the general theory for posterior contraction rates in the i.i.d.~sampling model of Ghosal et al.~\cite{GGvdV00}, asserting that if for some sequence $\xi_n\to0$ such that $n\xi_n^2\to\infty$, some constant $C>0$, and all $n\in\N$ large enough,
\begin{equation}
\label{Eq:KLCond}
	\Pi_n\left(p  :  -E_{p_0}\left(\log \frac{p}{p_0}(X)\right)\le  \xi_n^2,\ 
	E_{p_0}\left(\log \frac{p}{p_0}(X)\right)^2\le \xi_n^2\right)
	\ge  e^{-Cn\xi_n^2},
\end{equation}
and there exists sets of densities $\Pcal_n$ such that
\begin{equation}
\label{Eq:ApproxSetCond}
	\Pi_n(\Pcal_n^c)\le e^{-(C+4)n\xi_n^2},
\end{equation}
and
\begin{equation}
\label{Eq:EntropyCond}
	\log N(\xi_n; \Pcal_n, d_{TV})\lesssim n\xi_n^2,
\end{equation}
then, for sufficiently large $M>0$, $\Pi_n\left(p:d_{TV}(p,p_0)>M\xi_n\big|X^{(n)}\right)\to 0$ in $P_0$-probability as $n\to\infty$.

%
%
%
%
%

\subsection{Proof of Theorem \ref{Theo:ContrRateFixed}}
\label{Subsec:ProofTheoFixed}

Set $\xi_n := n^{-s/(2s + d)}$, and write $W_n$ in \eqref{Eq:WnFixed} as $W_n = (n\xi_n^2)^{-1}W$, where $W$ is the non-rescaled $(s-d)$-regular {Laplace} random element
\begin{equation}
\label{Eq:W}
	W := 
	\sum_{l=1}^\infty \sum_{r=1}^{2^{ld}} 2^{-l\left(s - \frac{d}{2}\right)}
	W_{lr}\psi_{lr},
	\qquad W_{lr}\iid \textnormal{Laplace}.
\end{equation}
Let $\Pi_{W}$ denote the law of $W$. We verify conditions \eqref{Eq:KLCond} - \eqref{Eq:EntropyCond} employing tools for Besov priors largely due to Agapiou et al.~\cite{ADH21}. Starting with \eqref{Eq:KLCond}, write $p_0 = f_0/\int_{\cube}f_0(x')dx'$ for some strictly positive $f_0\in B^s_{11}$. Since $\phi:\R\to(0,\infty)$ is strictly increasing, bijective and smooth, it {possesses} a strictly increasing, bijective and smooth inverse $\phi^{-1}:(0,\infty)\to\R$, so that $f_0 = \phi\circ w_0$ for $w_0 = \phi^{-1}\circ f_0\in B^s_{11}$. This follows from Theorem 10 in \cite{BS10}, upon writing $w_0(x) = \phi^{-1}((f_0(x) - \phi(0)) + \phi(0))$ and noting that $f_0 - \phi(0)\in B^s_{11}$ and that $\phi^{-1}(\cdot + \phi(0)):(-\phi(0),\infty) \to \R$ is smooth and vanishes at zero. Thus,
$$
	p_0(x) 
	= \frac{\phi(w_0(x))}{\int_{\cube}\phi(w_0(x'))dx'} = \phi_{w_0}(x), 
	\qquad x\in \cube.
$$
By construction (cf.~\eqref{Eq:PhiWn}), each density $p$ in the support of $\Pi_n$ takes the form $p = \phi_w$ for some $w\in C(\cube)$. For all such densities, since $\phi$ is uniformly log-Lipschitz, a standard computation (e.g., Problem 2.4 in \cite{GvdV17}) shows that for some $c_1>0$,
\begin{align*}
	\max\left\{-E_{p_0}\left(\log \frac{p}{p_0}(X)\right), \
	E_{p_0}\left(\log \frac{p}{p_0}(X)\right)^2\right\}
	&\lesssim \| w - w_0\|_\infty^2 e^{c_1 \| w - w_0\|_\infty^2}.
\end{align*}
Therefore, the prior probability in \eqref{Eq:KLCond} is lower bounded by
$$
	\Pi_{W_n}( w  :  \|w - w_0\|_\infty \le c_2 \xi_n )
	=\Pi_{W}\left(w: \|w - n\xi_n^2w_0\|_\infty \le c_2 n \xi_n^3
	\right )
$$
for some $c_2>0$. By construction, $W$ in \eqref{Eq:W} has associated `decentering' space $\Zcal = B^s_{11}$ with norm $\|\cdot\|_\Zcal = \|\cdot\|_{B^s_{11}}$ (cf.~Section \ref{Subsec:PropBesovPriors}). By the decentering inequality \eqref{Eq:Decentering}, it follows that the latter probability is greater than
$$
	e^{-\|n\xi_n^2w_0\|_\Zcal}
	\Pi_{W}\left(w:\|w\|_\infty\le c_2 n \xi_n^3\right)
	=e^{-\|w_0\|_{B^s_{11}}n\xi_n^2}
	\Pi_{W}\left(w:\|w\|_\infty\le c_2 n \xi_n^3\right).
$$
Finally, by the sup-norm small ball lower bound \eqref{Eq:MasterSmallBall}, applied with $t=s-d>0$, noting $n\xi_n^3 = n^{-(s - d)/(2s + d)}\to 0$ as $n\to\infty$, 
\begin{equation}
\label{Eq:SupNormSmallBall}
	\Pi_{W}\left(w:\|w\|_\infty\le c_2 n \xi_n^3\right)
	\ge 
		e^{-c_3(c_2n\xi_n^3)^{-d/(s - d)}}
	=
		e^{-c_4n^{d/(2s + d)}}
	=
		e^{-c_4n\xi_n^2},
\end{equation}
for some $c_3,c_4>0$. Combining the last two displays yields \eqref{Eq:KLCond} with $C= \|w_0\|_{B^s_{11}} + c_4$. Turning to conditions \eqref{Eq:ApproxSetCond} and \eqref{Eq:EntropyCond}, define the sieves $\Pcal_n := \left\{ \phi_w, \ w\in\Wcal_n\right\}$, where
\begin{align}
\label{Eq:Wcaln}
	\Wcal_n :=\left\{w = w^{(1)} + w^{(2)} : \|w^{(1)}\|_2\le R \xi_n,\  
	\|w^{(2)}\|_{B^s_{11}}\le R  \right\}
	\cap \{w  :  \|w\|_\infty\le R\}.
\end{align}
Then
\begin{align*}
	\Pi_n(\Pcal_n^c)
	&\le 
		\Pi_{W_n}(\Wcal_n^c)\\
	&\le 1-\Pi_{W_n}\Big( w= w^{(1)} + w^{(2)} 
	: \|w^{(1)}\|_2\le R \xi_n,  \ 
	\|w^{(2)}\|_{B^s_{11}}\le R  \Big)\\
	&\quad+ 1 - \Pi_{W_n}(w: \|w\|_\infty\le R ),
\end{align*}
which, by Lemma \ref{Lem:SievesFixed}, fixing $K>C+4$ and choosing sufficiently large $R>0$, is smaller than
$$
	2e^{-Kn\xi_n^2}\le e^{-(C+4)n\xi_n^2}.
$$
Finally, since $d_{TV}(p,p') = \frac{1}{2}\|p - p'\|_1$, $\log N(\xi_n; \Pcal_n, d_{TV})=
\log N(2\xi_n; \Pcal_n, \|\cdot\|_1)$. Also, by Lemma \ref{Lem:InfoIneq}, we have  for all $w,w'\in\Wcal_n$, for constants $c_5,c_6>0$,
$$
	\| \phi_w - \phi_{w'}\|_1
	\le \frac{c_5 e^{c_6\|w - w'\|_\infty}}{\phi(-\|w'\|_\infty)}
	\|w - w'\|_1' 
	\lesssim \|w - w'\|_1,
$$
since $\|w\|_\infty,\|w'\|_\infty\le R$. Thus, for some $c_7>0$,
$$
	\log N(2\xi_n; \Pcal_n, \|\cdot\|_1)
	\le \log N(c_7\xi_n; \Wcal_n, \|\cdot\|_1)
$$
and since $\Wcal_n \subset \left\{ w = w^{(1)} + w^{(2)} : \|w^{(1)}\|_2\le R \xi_n,\  
\|w^{(2)}\|_{B^s_{11}}\le R  \right\}$, using the continuous embedding of $B^s_{11}$ into $B^s_{1\infty}$ (e.g., Theorem 3.3.1 in \cite{T83}), the complexity bound stated in  Theorem 4.3.36 in \cite{GN16}, and the fact {that} $\|w^{(1)}\|_1\le \|w^{(1)}\|_2$, the latter metric entropy is bounded by a multiple of
$$
	\log N(\xi_n; \{w :  \|w\|_{B^s_{11}}\le R  \}, \|\cdot\|_1)
	\lesssim \xi_n^{-\frac{d}{s}} = n^\frac{d}{2s + d} = n\xi_n^2,
$$
concluding the verification of \eqref{Eq:EntropyCond} and the proof of Theorem \ref{Theo:ContrRateFixed}.

\qed
\endproof

	An inspection of the above proof reveals that Theorem \ref{Theo:ContrRateFixed} remains valid also if the total variation distance is replaced by either the Hellinger distance 
\begin{equation}
\label{Eq:HellDist}
	d_H(p,p') := \sqrt{ \int_{\cube}\left(\sqrt{p(x)} - \sqrt{p'(x)}\right)^2dx}
\end{equation}
or the $L^2$-distance. Indeed, for the Hellinger distance, the complexity bound \eqref{Eq:EntropyCond} can be verified with $d_{TV}$ replaced by $d_H$ noting that for $\Wcal_n$ as in \eqref{Eq:Wcaln}, the sieves $\Pcal_n=\left\{ \phi_w, \ w\in\Wcal_n\right\}$ contain densities that are uniformly bounded and bounded away from zero, whence the equivalence $d_H(\phi_w,\phi_{w'})\simeq \|\phi_w - \phi_{w'}\|_2$ for all $w,w'\in \Wcal_n$ (following, e.g., from Lemma B.1 in \cite{GvdV17}). Similar computations as in the proof of Lemma \ref{Lem:InfoIneq} further show that $\|\phi_w - \phi_{w'}\|_2\lesssim \|w - w'\|_2$ for all $w,w'\in \Wcal_n$, so that
$$
	\log N(\xi_n; \Pcal_n, d_H)
	\lesssim
	\log N(\xi_n; \{w :  \|w\|_{B^s_{11}}\le R  \}, \|\cdot\|_2)
	\lesssim  \xi_n^{-\frac{d}{s}} = n\xi_n^2.
$$
Via Theorem 2.1 in \cite{GGvdV00}, this yields
$$
	\Pi_n\left( p  : \ d_H(p,p_0) \le M n^{-\frac{s}{2s + d}} \Big| X^{(n)}\right) \to 1
$$
in $P_0$-probability as $n\to\infty$. For the $L^2$-distance, note that via Theorem 8.20 in \cite{GvdV17}, using the conditions \eqref{Eq:KLCond} and \eqref{Eq:ApproxSetCond} verified above, the last display can be strengthened to
$$
	\Pi_n\left( \phi_w\in \Pcal_n  : \ d_H(\phi_w,p_0)
	\le M n^{-\frac{s}{2s + d}} \Big| X^{(n)}\right) \to 1
$$
in $P_0$-probability as $n\to\infty$, implying posterior contraction in $L^2$-distance since $\|\phi_w - p_0\|_2\lesssim d_H(p,p_0)$ for all $w\in\Wcal_n$.

\begin{lemma}\label{Lem:SievesFixed}
	For $s>d$, let $\Pi_{W_n}$ be the rescaled {Laplace} prior arising as the law of $W_n$ in \eqref{Eq:WnFixed}. Then, for all $K>0$ there exist sufficiently large $R>0$ such that, for $n\in\N$ large enough,
\begin{enumerate}
\item
\begin{align*}
	\Pi_{W_n}\left( w= w^{(1)} + w^{(2)} 
	: \|w^{(1)}\|_2\le R n^{-\frac{s}{2s+d}}, \ 
	\|w^{(2)}\|_{B^s_{11}}\le R  \right)
	&\ge 1 - e^{-K n^{d/(2s + d)}};
\end{align*}
\item
\begin{align*}
	\Pi_{W_n}(w: \|w\|_\infty\le R )
	&\ge 1 - e^{-K n^{d/(2s + d)}}.
\end{align*}

\end{enumerate}

\end{lemma}

\begin{proof}
	To prove point 1., letting $\xi_n$, $W$ and $\Pi_{W}$ be defined as at the beginning of the proof of Theorem \ref{Theo:ContrRateFixed}, the probability of interest equals
\begin{equation}
\label{Eq:ProbInt}
	\Pi_{W}\left( w= w^{(1)} + w^{(2)} 
	 :  \|w^{(1)}\|_2\le R n\xi_n^3,\  
	\|w^{(2)}\|_{B^s_{11}}\le R n\xi_n^2  \right).
\end{equation}
By construction, the spaces associated to $W$ in \eqref{Eq:W} are respectively $\Zcal = B^s_{11}$, with norm $\|\cdot\|_\Zcal = \|\cdot\|_{B^s_{11}}$, and $\Qcal = H^{s-d/2} $, $\|\cdot\|_\Qcal = \|\cdot\|_{H^{s-d/2}}$; see Section \ref{Subsec:PropBesovPriors}. By the two-level concentration inequality \eqref{Eq:TwoLevConc} it follows that for some $c_1>0$, for all $\overline R>0$,
\begin{align*}
	\Pi_{W}\Big( w= \overline w^{(1)} + \overline w^{(2)} + &\overline w^{(3)}
	 :  \|\overline w^{(1)}\|_2\le n\xi_n^3, \ \| \overline w^{(2)}\|_{H^{s-d/2}}
	\le \sqrt{ \overline R n\xi_n^2},\\ 
	& 
	\|\overline w^{(3)}\|_{B^s_{11}}\le \overline R n\xi_n^2  \Big)
	\ge 1 - \frac{1}{\Pi_{W}\left(w:\|w\|_2\le n\xi_n^3\right)}
	e^{-c_1\overline Rn\xi_n^2}.
\end{align*}
Since $\|w\|_2\le \|w\|_\infty$, using \eqref{Eq:SupNormSmallBall}, we have that $\Pi_{W}\left(w:\|w\|_2\le n\xi_n^3\right)\ge e^{-c_2n\xi_n^2}$ for some $c_2>0$ as $n\to\infty$, so that the right hand side of the last display is lower bounded by 
\begin{equation}
\label{Eq:ExpIneq}
	1 -e^{-(c_1 \overline R - c_2)n\xi_n^2} \ge 1 -e^{-Kn\xi_n^2}
\end{equation}
upon choosing sufficiently large $\overline R>0$. Now for $\overline w^{(2)}$ as in the second to last display, let $P_{L_n}\overline w^{(2)}$ be its wavelet approximation at resolution $L_n\in\N$ satisfying $2^{L_n}\simeq n^{1/(2s + d)}$; see Section \ref{Subsec:Preliminaries} for definitions. Then, by the properties of wavelet projections,
$$
	\|\overline w^{(2)} - P_{L_n}\overline w^{(2)}\|_2 \le 2^{-L_n\left(s-\frac{d}{2}\right)}
	\|\overline
	w^{(2)}\|_{H^{s-d/2}}
	\lesssim  n^{-\frac{s-d/2}{2s + d}}n^\frac{d/2}{2s + d}
	=n\xi_n^3,
$$
and moreover, by the wavelet characterisation of $\|\cdot\|_{B^s_{11}}$ (cf.~Section \ref{Subsec:Preliminaries}) and H\"older's inequality,
\begin{align*}
	\|P_{L_n}\overline w^{(2)}\|_{B^s_{11}}
	&=
	\sum_{l=1}^{L_n}2^{l\left(s -\frac{d}{2}\right)}\sum_{r=1}^{2^{ld}}|\langle
	\overline w^{(2)},\psi_{lr}\rangle_2|  \\
	&\le
	\sqrt{\textnormal{dim}(V_{L_n})}\sqrt {\sum_{l=1}^{L_n}
	2^{2l\left(s -\frac{d}{2}\right)}\sum_{r=1}^{2^{ld}}|\langle
	\overline w^{(2)},\psi_{lr}\rangle_2|^2}\\
	&\lesssim \sqrt{2^{L_nd}}\|\overline 
	w^{(2)}\|_{H^{s-d/2}}
	\lesssim n^\frac{d/2}{2s+d} n^\frac{d/2}{2s+d} = n\xi_n^2.
\end{align*}
Taking $w^{(1)} := \overline w^{(1)} + (\overline w^{(2)} - P_{L_n}\overline w^{(2)})$ and $w^{(2)} := P_{L_n}\overline w^{(2)} + \overline w^{(3)}$ thus implies $\|w^{(1)}\|_2\lesssim  n\xi_n^3$ and $ \|w^{(2)}\|_{B^s_{11}}\lesssim n\xi_n^2$, so that taking $R>0$ large enough the probability of interest \eqref{Eq:ProbInt} is lower bounded by the right hand side of \eqref{Eq:ExpIneq}, concluding the proof of the first claim.

	To prove point 2., the probability of interest equals $\Pi_{W} \left(w:\|w\|_\infty\le R n\xi_n^2 \right)$, to which we directly apply the concentration inequality \eqref{Eq:SupNormConc} to deduce the lower bound, holding for $R>0$ large enough,
$$
	1 - c_3 e^{-c_4 R n\xi_n^2}
	\ge 1 - e^{-K n\xi_n^2}.
$$
\end{proof}

%
%
%
%
%

\subsection{Proof of Theorem \ref{Theo:ContrRatePartial}}
\label{Subsec:ProofTheoPartial}

We verify conditions \eqref{Eq:KLCond} - \eqref{Eq:EntropyCond} with $\xi_n := n^{-s/(2s + d)}\sqrt {\log n}$. Since $\phi:\R\to(B,\infty)$ is strictly increasing, bijective and smooth, it posses a strictly increasing, bijective and smooth inverse $\phi^{-1}:(B,\infty)\to\R$. Hence, 
for $p_0 \propto f_0$ with $f_0\in B^s_{11} $ satisfying $f_0(x) > B$ for all $x\in \cube$, we have $p_0 = \phi_{w_0}$ with $w_0 = \phi^{-1}\circ f_0\in B^s_{11} $ (following again by Theorem 10 in \cite{BS10} applied to the smooth function $\phi^{-1}(\cdot + \phi(0))$ vanishing at zero). Since $\phi$ is uniformly Lipschitz, by point 1.~in Lemma \ref{Lem:InfoIneqLip}, for each $w\in C(\cube)$ such that $\| w - w_0\|_\infty\le \xi_n$, since then $\|w\|_\infty\le \|w_0\|_\infty + 1$, we have
\begin{align*}
	\max\Bigg\{&-E_{p_0}\left(\log \frac{\phi_w}{p_0}(X)\right), 
	E_{p_0}\left(\log \frac{\phi_w}{p_0}(X)\right)^2\Bigg\}\\
	&\lesssim
	\frac{1}{B^2}\left\|\frac{p_0}{\phi_w}\right\|_\infty
	\left\|  w - w_0\right\|_2^2
	\le \frac{\|p_0\|_\infty}{B^3/\phi(\|w_0\|_\infty+1)} \|  w - w_0\|^2_2
	\lesssim \|  w - w_0\|^2_\infty.
\end{align*}
Thus, for some $c_1>0$,
\begin{align*}
	\Pi_n\Bigg(p  :  -E_{p_0}\left(\log \frac{p}{p_0}(X)\right)\le  \xi_n^2,&\ 
	E_{p_0}\left(\log \frac{p}{p_0}(X)\right)^2\le \xi_n^2\Bigg)\\
	&\ge \Pi_{W_n}(w  :  \|w - w_0\|_\infty\le c_1\xi_n).
\end{align*}
By construction, the decentering space $\Zcal_n$ associated to $W_n$ in \eqref{Eq:WnFixedPartial} satisfies $\Zcal_n = B^s_{11} $, with norm
$$
	\|w\|_{\Zcal_n} = \sum_{l=1}^{L_n}\sum_{r=1}^{2^{ld}} 2^{l\left(s - \frac{d}{2}\right)}
	|\langle w,\psi_{lr}\rangle_2|
	+n\xi_n^2
	\sum_{l=1}^\infty \sum_{r=1}^{2^{ld}} 2^{l\left(s - \frac{d}{2}\right)}
	|\langle w,\psi_{lr}\rangle_2|
	\le
	n\xi_n^2\|w\|_{B^s_{11}}.
$$
Thus by the decentering inequality \eqref{Eq:Decentering}, the right hand side of the second to last display is lower bounded by $e^{-\|w_0\|_{B^s_{11}}n\xi_n^2}\Pi_{W_n}(w:\|w\|_\infty\le c_1 \xi_n).$ Decompose $W_n$ in \eqref{Eq:WnFixedPartial} as $W_n = W_n^{(1)} + (n\xi_n^2)^{-1}W_n^{(2)}$, where
\begin{equation}
\label{Eq:W1W2}
	W_n^{(1)}
	:= 
	\sum_{l=1}^{L_n}\sum_{r=1}^{2^{ld}} 2^{-l\left(s - \frac{d}{2}\right)}
	W_{lr}\psi_{lr};
	\ \
	W_n^{(2)}:=
	\sum_{l=L_n+1}^\infty \sum_{r=1}^{2^{ld}} 2^{-l\left(s - \frac{d}{2}\right)}
	W_{lr}\psi_{lr},
	\ \ W_{lr}\iid \textnormal{Laplace},
\end{equation}
and note that by construction $W_n^{(1)}$ and $W_n^{(2)}$ are independent. Then,\begin{align*}
	\Pi_{W_n}(w:\|w\|_\infty\le c_1 \xi_n)
	&\ge\Pr\Big(\|W_n^{(1)}\|_\infty\le \frac{c_1}{2} \xi_n,\   
	\|(n\xi_n^2)^{-1}W_n^{(2)}\|_\infty\le \frac{c_1}{2} \xi_n\Big)\\
	&=
	\Pr\left(\|W_n^{(1)}\|_\infty\le \frac{c_1}{2} \xi_n\right)
	\Pr\left(\|W^{(2)}_n\|_\infty\le \frac{c_1}{2} n\xi_n^3\right).
\end{align*}
By the continuous embedding $B^s_{11} \subset C(\cube)$, holding since $s>d$ (e.g., \cite{GN16}, p.370), the first probability is lower bounded, for some $c_2>0$, by
\begin{align*}
	\Pr\left(\|W_n^{(1)}\|_{B^s_{11}}\le c_2 \xi_n\right)
	&=
		\Pr\left( \sum_{l=1}^{L_n}\sum_{r=1}^{2^{ld}}|W_{lr}|\le c_2\xi_n\right)\\
	&\ge
		\Pr\left(\textnormal{dim}(V_{L_n})\max_{1\le l\le L_n}\max_{r=1,\dots,2^{ld}}|W_{lr}|
		\le c_2\xi_n\right)\\
	&=
		\prod_{l =1}^{L_n}\prod_{r=1}^{2^{ld}}\Pr\left(|W_{lr}|
		\le c_2\frac{\xi_n}{\textnormal{dim}(V_{L_n})} \right).
\end{align*}
Since for all $z\in(0,1)$, $\Pr(|W_{lr}|\le z) = 1 - e^{-z}\ge z/2$, recalling $\textnormal{dim}(V_{L_n})=O(2^{L_nd})=O(n^{d/(2s+d)})$, the last line is greater than
$$
	\left( c_3n^{-\frac{s+d}{2s+d}}\sqrt{\log n}\right)^{c_4n^{d/(2s+d)}}
	\ge e^{c_4 n^{d/(2s+d)}\log( n^{-c_5})}
	= e^{-c_6 n^{d/(2s+d)} \log n}
	= e^{-c_6n\xi_n^2},
$$
for $c_3,\dots,c_6>0$. On the other hand, by the small ball lower bound \eqref{Eq:MasterSmallBall}, which also applies with the first $L_n$ terms in the series removed, with the choice $t=s-d$,
\begin{align*}
	\Pr\Big(\|W^{(2)}_n\|_\infty\le \frac{c_1}{2} &n\xi_n^3\Big)\ge
		\Pr\left(\|W^{(2)}_n\|_\infty\le \frac{c_1}{2} n^{-\frac{s-d}{2s+d}}\right)
	\ge
		e^{-c_7 n^{d/(2s+d)}}
	\ge 
		e^{-c_7n\xi_n^2}
\end{align*}
as $n\to\infty$ for some $c_7>0$. Combining the last two displays yields condition \eqref{Eq:KLCond} with $C=\|w_0\|_{B^s_{11}}+c_6+c_7$. Next, define the sieves $\Pcal_n := \left\{ \phi_w, \ w\in\Wcal_n\right\}$, where
$$
	\Wcal_n = \left\{ w = w^{(1)} + w^{(2)}  
	 :  w^{(1)} \in\Wcal_n^{(1)}, \ w^{(2)} \in\Wcal_n^{(2)} \right\}
$$
with
$$
	\Wcal_n^{(1)} = \{ w^{(1)}\in V_{L_n}  :  \|w^{(1)}\|_2\le Rn\xi_n^2\}
$$
and
\begin{align*}
	\Wcal_n^{(2)} 
	&=\left\{ w^{(2)} = w^{(2,1)} + w^{(2,2)} 
	: \|w^{(2,1)}\|_2\le R n^{-\frac{s}{2s+d}}, \  
	\|w^{(2,2)}\|_{B^s_{11}}\le R  \right\}\\
	&\quad \cap \left\{w^{(2)}  :  \|w^{(2)}\|_\infty\le R\right\}.
\end{align*}
Condition \eqref{Eq:ApproxSetCond} then follows since
\begin{align*}
	\Pi_n(\Pcal_n^c)
	&\le 
		\Pi_{W_n}(\Wcal_n^c)\\
	&\le
		\Pr\left(W_n^{(1)}\notin\Wcal_n^{(1)}\right) 
		+ \Pr\left((n\xi_n^2)^{-1}W_n^{(2)}\notin\Wcal_n^{(2)}\right)\\
	&\le 
	3 - \Pr\left(W^{(1)}_n\in V_{L_n},\ \|W^{(1)}_n\|_2\le Rn\xi_n^2\right )
	-
	\Pr\Big( W^{(2)}_n= W^{(2,1)}_n + W^{(2,2)}_n 
	: \\
	&\quad \ \|W^{(2,1)}_n\|_2\le R n^{-\frac{s-d}{2s+d}}\log n, \  
	\|W^{(2)}_n\|_{B^s_{11}}\le R n\xi_n^2  \Big)
	- \Pr\left( \|W_n^{(2)}\|_\infty\le Rn\xi_n^2 \right)
\end{align*}
which, by Lemma \ref{Lem:SievesFixedPartial}, for $n\in\N$ large enough, fixing $K>C+4$ and choosing sufficiently large $R>0$, is smaller than
$$
	3e^{-Kn\xi_n^2}\le e^{-(C+4)n\xi_n^2}.
$$
Finally, let $\{\gamma_j, \ j=1,\dots,J_n\}$, with $J_n = N(\xi_n;\Wcal_n^{(1)},\|\cdot\|_2)$, be a minimal $\xi_n$-covering of $\Wcal_n^{(1)}$ in $L^2$-distance, and let $\{\chi_k, \ k=1,\dots,K_n\}$, with $K_n = N(\xi_n;\Wcal_n^{(2)},\|\cdot\|_1)$, be a minimal $\xi_n$-covering of $\Wcal_n^{(2)}$ in $L^1$-distance. Then for each $w = w^{(1)} + w^{(2)}\in \Wcal_n$ there exist $j\in\{1,\dots,J_n\}$ and $k\in\{1,\dots, K_n\}$ such that
$$
	\| w^{(1)} - \gamma_j\|_2\le \xi_n; \qquad \| w^{(2)} - \chi_k \|_1 \le \xi_n,
$$
which by point 2.~in  Lemma \ref{Lem:InfoIneqLip} implies that
\begin{align*}
	d_{TV}(\phi_w,\phi_{\gamma_j + \chi_k})
	\lesssim \|w - (\gamma_j + \chi_k)\|_1
	\le \| w^{(1)} - \gamma_j\|_2 +  \| w^{(2)} - \chi_k \|_1
	\lesssim \xi_n.
\end{align*}
It follows that for some $c_8>0$ the set $\{\gamma_j + \chi_k, \ j=1,\dots,J_n, \ k=1,\dots,K_n\}$ is a $c_8\xi_n$-covering of $\Wcal_n$ in total variation distance, whence
\begin{equation}
\label{Eq:MetrEntr}
	\log N(\xi_n;\Wcal_n,d_{TV})
	\lesssim \log(J_nK_n)
	=\log  N(\xi_n;\Wcal_n^{(1)},\|\cdot\|_2)
	+\log N(\xi_n;\Wcal_n^{(2)},\|\cdot\|_1).
\end{equation}
Recalling $\Wcal_n^{(1)}\subset V_{L_n}$, with $\textnormal{dim}(V_{L_n}) = O(n^\frac{d}{2s + d})$, the first metric entropy equals
\begin{align*}
	\log N(\xi_n;\{ y\in \R^{\textnormal{dim}(V_{L_n})},
	\ |y|_2\le Rn\xi_n^2\},
	|\cdot|_2)
	&\le 
		\textnormal{dim}(V_{L_n})\log \left(\frac{3Rn\xi_n^2}{\xi_n}\right)\\
	&\lesssim
		n^\frac{d}{2s + d}\log (n^{c_9} )
	\simeq 
		n\xi_n^2,
\end{align*}
having used the usual metric entropy bound for balls in Euclidean spaces (e.g., \cite[Theorem 4.3.34]{GN16}). Arguing as in the conclusion of the proof of Theorem \ref{Theo:ContrRateFixed}, the second metric entropy in \eqref{Eq:MetrEntr}, is bounded by a multiple of
$$
	\log N(\xi_n; \{w :  \|w\|_{B^s_{11}}\le R  \}, \|\cdot\|_1)
	\lesssim \xi_n^{-\frac{d}{s}} = n^\frac{d}{2s + d}(\log n)^{-\frac{d}{2s}} 
	\le  n\xi_n^2.
$$
The last two displays and \eqref{Eq:MetrEntr} conclude the verification of \eqref{Eq:EntropyCond} and the proof of Theorem \ref{Theo:ContrRatePartial}.

\qed
\endproof

\begin{lemma}\label{Lem:SievesFixedPartial}
	For $s>d$, let $W_n^{(1)},W_n^{(2)}$ be the random functions in \eqref{Eq:W1W2}. Then, for all $n\in\N$ large enough, all $K>0$, there exist sufficiently large $R>0$ such that
\begin{enumerate}
\item
$$
	\Pr\left( W_n^{(1)} \in V_{L_n} , \ \|W_n^{(1)}\|_\infty\le Rn^\frac{d}{2s+d}\log n\right)
	\ge 1 - e^{-K n^{s/(2s + d)}\log n};
$$
\item
\begin{align*}
	\Pr\Big( W^{(2)}_n= W^{(2,1)}_n + W^{(2,2)}_{n} 
	 :  &\|W^{(2,1)}_n\|_2\le R n^{-\frac{s-d}{2s+d}}\sqrt{\log n},\\  
	&\|W^{(2)}_n\|_{B^s_{11}}\le R n^\frac{d}{2s + d}\log n \Big)
	\ge 1 - e^{-K n^{d/(2s + d)}\log n};
\end{align*}
\item
$$
	\Pr\left( \|W_n^{(2)}\|_\infty\le Rn^\frac{d}{2s+d}\log n\right)
	\ge 1 - e^{-K n^{d/(2s + d)}\log n}.
$$
\end{enumerate}
\end{lemma}

\begin{proof}
For point 1., setting $\xi_n := n^{-s/(2s + d)}\sqrt{\log n}$, recalling $W_n^{(1)}\in V_{L_n}$ by construction, the probability of interest equals
$$
	\Pr\left( \|W_n^{(1)}\|_\infty\le Rn\xi_n^2\right)
	\ge 1 - e^{-K n\xi_n^2}
$$
having used  the sup-norm concentration inequality \eqref{Eq:SupNormConc} and chosen $R>0$ large enough. Points 2.~and 3.~follow from exactly the same argument as in the proof of Lemma \ref{Lem:SievesFixed}, upon noting that $W_n^{(2)}$ is also a $(s-d)$-regular Besov random element with associated spaces
$$
	\Qcal_n := \left\{w= \sum_{l=L_n+1}^\infty\sum_{r = 1}^{2^{ld}}w_{lr}\psi_{lr} :  
	\|w\|_{\Qcal_n}=\|w\|_{H^{s-d/2}}<\infty
	\right\}, 
$$
and
$$
	\Zcal_n := \left \{w= \sum_{l=L_n+1}^\infty\sum_{r = 1}^{2^{ld}}w_{lr}\psi_{lr} :  
	\|w\|_{\Zcal_n}=\|w\|_{B^s_{11}}<\infty
	\right\},
$$
and that for each $w\in\Qcal_n$,
\begin{align*}
	\|w\|_{L^2}^2&=\sum_{l=L_n+1}^\infty\sum_{r = 1}^{2^{ld}}
	2^{-2l\left(s-\frac{d}{2}\right)}2^{2l\left(s-\frac{d}{2}\right)}|w_{lr}|^2\\
	&\le 2^{-2L_n\left(s-\frac{d}{2}\right)}\| w \|^2_{H^{s-d/2}}
	\lesssim n^{-\frac{2s-d}{2s+d}}\| w \|^2_{H^{s-d/2}}.
\end{align*}
The details are omitted for brevity.

\end{proof}

%
%
%
%
%

\subsection{Proof of Theorem \ref{Theo:ContrRateTrunc}}
\label{Subsec:ProofTheoTrunc}

We verify conditions \eqref{Eq:KLCond} - \eqref{Eq:EntropyCond} with $\xi_n := c_1n^{-s/(2s+d)}$, for sufficiently large $c_1>0$ to be chosen below. Arguing as in the proof of Theorem \ref{Theo:ContrRateFixed}, 
\begin{align*}
	\Pi_n\Bigg(p  :  -E_{p_0}\left(\log \frac{p}{p_0}(X)\right)\le  \xi_n^2,&\ 
	E_{p_0}\left(\log \frac{p}{p_0}(X)\right)^2\le \xi_n^2\Bigg)\\
	&\ge\Pi_{W_n}( w  :  \|w - w_0\|_\infty \le c_2 \xi_n )
\end{align*}
for some $c_2>0$ and $w_0 = \phi^{-1}\circ f_0\in B^s_{\infty\infty} $. The decentering space $\Zcal_n$  associated to $W_n$ in \eqref{Eq:WnFixedTrunc} equals the approximation space $ V_{L_n}$, with norm
$\|\cdot\|_{\Zcal_n} = \|\cdot\|_{B^{s+d}_{11}}$. Since $w_0\in B^s_{\infty\infty} $, its wavelet projection $P_{L_n}w_0\in V_{L_n}$ satisfies
$$
	\|w_0 - P_{L_n}w_0\|_\infty\le 2^{-L_ns}\|w_0\|_{B^s_{\infty\infty}}
	\simeq \|w_0\|_{B^s_{\infty\infty}} n^{-\frac{s}{2s+d}}
$$
as well as
\begin{align*}
	\|P_{L_n}w_0\|_{B^{s+d}_{11}}
	&=\sum_{l=1}^{L_n}2^{l\left(s+\frac{d}{2}\right)}\sum_{r=1}^{2^{ld}}
	|\langle w_0,\psi_{lr} \rangle_2|\\
	&\le
	\|w_0\|_{B^s_{\infty\infty}}\sum_{l=1}^{L_n}2^{ld}
	\simeq \|w_0\|_{B^s_{\infty\infty}}2^{L_nd}
	\simeq\|w_0\|_{B^s_{\infty\infty}}n\xi_n^2.
\end{align*}
By the triangle and the decentering inequality \eqref{Eq:Decentering}, taking $c_1>0$ large enough in the definition of $\xi_n$, for some $c_3,c_4>0$,
\begin{align*}
	\Pi_{W_n}( w  :  \|w - w_0\|_\infty \le c_2 \xi_n )
	&\ge\Pi_{W_n}( w  :  \|w - P_{L_n}w_0\|_\infty \le c_3 \xi_n )\\
	&\ge
	e^{-c_4\|w_0\|_{B^s_{\infty\infty}}n\xi_n^2}
	\Pi_{W_n}( w  :\|w\|_\infty\le c_3 \xi_n).
\end{align*}
By another application of the small ball estimate \eqref{Eq:MasterSmallBall}, now with $t = s$, we have
\begin{align}
\label{Eq:SmallBallTrunc}
	\Pi_{W_n}( w  :\|w\|_\infty\le c_3 \xi_n)
	\ge
	e^{-c_5(c_3\xi_n)^{-d/s}} = e^{-c_6 n\xi_n^2},
\end{align}
for $c_5,c_6>0$ as $n\to\infty$, whence condition \eqref{Eq:KLCond} follows for $C=c_4\|w_0\|_{B^s_{\infty\infty}}+c_6$. Next, take the sieves $\Pcal_n := \left\{ \phi_w, \ w\in\Wcal_n\right\}$, where
\begin{align*}
	\Wcal_n :=\left\{ w = w^{(1)} + w^{(2)} : \|w^{(1)}\|_\infty\le R \xi_n,\  
	\|w^{(2)}\|_{B^{s+d}_{11}}\le Rn\xi_n^2  \right\}.
\end{align*}
Using Lemma \ref{Lem:SievesFixedTrunc}, choosing $R>0$ large enough,
\begin{align*}
	\Pi_n(\Pcal_n^c)
	&\le 1-\Pi_{W_n}( w= w^{(1)} + w^{(2)} 
	: \|w^{(1)}\|_\infty\le R \xi_n, \  
	\|w^{(2)}\|_{B^{s+d}_{11}}\le R n\xi_n^2 \Big)\\
	&\le e^{-(C+4)n\xi_n^2}
\end{align*}
verifying condition \eqref{Eq:ApproxSetCond}. Finally, since the link function $\phi$ is uniformly log-Lipschitz, by Problem 2.4 in \cite{GvdV17}, for all $w,w'\in\Wcal_n$,
$$
	d_{TV}(\phi_w,\phi_{w'})\le 2d_H(\phi_w,\phi_{w'})\lesssim \|w - w'\|_\infty e^{c_7
	\|w - w'\|_\infty},
$$
for some $c_7>0$. As $\xi_n\to0$, for sufficiently large $n$, the $\xi_n$-entropy of $\Pcal_n$ in total variation distance is then bounded by the $c_8\xi_n$-entropy of $\Wcal_n$ in sup-norm distance for some $c_8>0$,
$$
	\log N(\xi_n;\Pcal_n,d_{TV})
	\le \log N(c_8\xi_n;\Wcal_n,\|\cdot\|_\infty).
$$
Using the continuous embedding $B^{s+d}_{11}\subset B^{s+d}_{1\infty}$ (e.g., Theorem 3.3.1 in \cite{T83}) and Theorem 4.3.36 in \cite{GN16}, the latter metric entropy is bounded by a multiple of
$$
	\log N\left(\xi_n; \{w :  \|w\|_{B^{s+d}_{11}}\le Rn\xi_n^2  \}, \|\cdot\|_\infty\right)
	\lesssim (n \xi_n)^{\frac{d}{s+d}} \simeq n\xi_n^2,
$$
concluding the verification of \eqref{Eq:EntropyCond} and the proof of Theorem \ref{Theo:ContrRateTrunc}.

\qed
\endproof

\begin{lemma}\label{Lem:SievesFixedTrunc}
	For $s>d$, let $\Pi_{W_n}$ be the non-rescaled truncated Laplace prior arising as the law of $W_n$ in \eqref{Eq:WnFixedTrunc}. Then, for all $n\in\N$ large enough, all $K>0$, there exist sufficiently large $R>0$ such that
\begin{align*}
	\Pi_{W_n}\Big( w= w^{(1)} + w^{(2)} 
	 :  \|w^{(1)}\|_\infty\le R n^{-\frac{s}{2s+d}},\   
	\|w^{(2)}\|_{B^{s+d}_{11}}\le &R n^\frac{d}{2s + d} \Big)\\
	&\ge 1 - e^{-K n^{d/(2s + d)}}.
\end{align*}
\end{lemma}

\begin{proof}
	The claim follows similarly to the proof of point 1.~in Lemma \ref{Lem:SievesFixed} and point 2.~in Lemma \ref{Lem:SievesFixedPartial}, using the two-level concentration inequality \eqref{Eq:TwoLevConc}, the centred small ball estimate  \eqref{Eq:MasterSmallBall} and noting that	the spaces associated to $W_n$ satisfy $\Zcal_n = \Qcal_n = V_{L_n}$ with norms $\|\cdot\|_{\Zcal_n} = \|\cdot\|_{B^{s+d}_{11}}$ and $\|\cdot\|_{\Qcal_n} = \|\cdot\|_{H^{s+d/2}}$, and further that for each $w\in V_{L_n}$ satisfying $\|w\|_{H^{s+d/2}}\lesssim n^{d/(4s+2d)}$,
\begin{align*}
	\|w\|_{B^{s+d}_{11}}
	&=
	\sum_{l=1}^{L_n}2^{l\left(s +\frac{d}{2}\right)}\sum_{r=1}^{2^{ld}}|w_{lr}| \\
	&\le
	\sqrt{\textnormal{dim}(V_{L_n})}\sqrt {\sum_{l=1}^{L_n}
	2^{2l\left(s +\frac{d}{2}\right)}\sum_{r=1}^{2^{ld}}|w_{lr}|^2} 
	\simeq n^\frac{d/2}{2s+d} \|w\|_{H^{s+d/2}}
	= n^\frac{d}{2s + d}.
\end{align*}
\end{proof}

%
%
%
%
%

\subsection{Proof of Theorem \ref{Theo:ContrRateAdapt}}
\label{Subsec:ProofTheoAdapt}

We verify conditions \eqref{Eq:KLCond} - \eqref{Eq:EntropyCond} with $\xi_n := c_1n^{-s_0/(2s_0 + d)}$ for sufficiently large $c_1>0$ to be chosen below. Arguing as in the proof of Theorem \ref{Theo:ContrRatePartial}, 
\begin{align*}
	\Pi_n\Bigg(p  :  -E_{p_0}\left(\log \frac{p}{p_0}(X)\right)\le  \xi_n^2,&\ 
	E_{p_0}\left(\log \frac{p}{p_0}(X)\right)^2\le \xi_n^2\Bigg)\\
	&\ge \Pi_{W_n}(w  :  \|w - w_0\|_\infty\le c_2\xi_n)
\end{align*}
for some $c_2>0$ and $w_0 = \phi^{-1}\circ f_0\in B^{s_0}_{11} $. Condition \eqref{Eq:KLCond} then follows from Lemma \ref{Lem:SmallBallHier} for a large enough constant $C>0$, upon taking sufficiently large $c_1$ in the definition of $\xi_n$.

	Next, define the sieves $\Pcal_n = \{\phi_w, \ w\in\Wcal_n\}$, where
$$
	\Wcal_n := \left\{w = w^{(1)} +  w^{(2)} : \|w^{(1)}\|_ 1
	 \le Rn^{-\frac{s^*}{2s^*+d}},\ \|  w^{(2)}\|_{B^{s^*+d}_{11}}
	\le R n^\frac{d}{2s^*+d}\right\},
$$
with $s^*:=s_0/(1 + M/\log n)$ and $R,M>0$. By Lemma \ref{Lem:SievesHier}, taking sufficiently large $M$ and $R$, for all $n$ large enough,
$$
	\Pi_n(\Pcal_n^c)\le\Pi_{W_n}(\Wcal_n^c)\le e^{-(C+4)n\xi_n^2},
$$
which verifies Condition \eqref{Eq:ApproxSetCond}. Finally, recalling that $\phi$ is assumed to be uniformly Lipschitz and bounded away from zero, by point 2.~in Lemma \ref{Lem:InfoIneqLip} we have $d_{TV}(\phi_w,\phi_{w'})\lesssim \|w - w'\|_1$, implying that
$\log N(\xi_n;\Pcal_n,d_{TV})\le \log N(c_3\xi_n;\Wcal_n,\|\cdot\|_1)$ for some $c_3>0$. By construction of $\Wcal_n$, using \eqref{Eq:Equiv} below and Theorem 4.3.36 in \cite{GN16}, the latter metric entropy is bounded by a multiple of
\begin{align*}
	\log N\Big(\xi_n; &\left\{ w: \|w\|_{B^{s^*+d}_{11}}
	\le R n^\frac{d}{2s^*+d} \right\}, \|\cdot\|_1\Big)\\
	&\lesssim 
	\left( \frac{R n^\frac{d}{2s^*+d} }{\xi_n}\right)^{\frac{d}{s^*+d}}
	\lesssim 
	\left(n^\frac{d}{2s^*+d}n^\frac{s^*}{2s^*+d}\right)^{\frac{d}{s^*+d}}
	= 
	n^\frac{d}{2s^*+d}
	\lesssim n\xi_n^2.
\end{align*}
This concludes the verification of Condition \eqref{Eq:EntropyCond}, and via an application of Theorem 2.1 in \cite{GGvdV00}, the proof of Theorem \ref{Theo:ContrRateAdapt}.

\qed
\endproof

\begin{lemma}\label{Lem:SmallBallHier}
Let $\Pi_{W_n}$ be the hierarchical rescaled {Laplace} prior arising as the law of $W_n$ in \eqref{Eq:WnHierarc}. Let $w_0\in B^{s_0}_{11}(\cube)$, any $s_0>d$. Then, for sufficiently large $D_1, D_2>0$,
$$
	\Pi_{W_n}\left(w:\|w - w_0\|_\infty \le D_1n^{-\frac{s_0}{2s_0+d}}
	\right)\ge e^{-D_2n^{d/(2s_0 + d)}}.
$$
\end{lemma}

\begin{proof}
For each fixed $s>d$, let $\varepsilon_{s,n} := n^{-s/(2s+d)}$ and let $\Pi_{W_{s,n}}$ be the rescaled $(s-d)$-regular {Laplace} prior arising as the law of
\begin{equation}
\label{Eq:Wsn}
	W_{s,n} 
	:= 
	\frac{W_s}{n\varepsilon_{s,n}^2},
	\qquad W_s:=
	\sum_{l=1}^\infty \sum_{r=1}^{2^{ld}} 2^{-l\left(s - \frac{d}{2}\right)}
	W_{lr}\psi_{lr},
	\qquad W_{lr}\iid \textnormal{Laplace}.
\end{equation}
Denote by $\Pi_{W_s}$ the law of $W_s$. The probability of interest equals
\begin{align}
\label{Eq:HierProb}
	\int_d^{\log n}\Pi_{W_{s,n}}(w:\|w - w_0\|_\infty & \le D_1\varepsilon_{s_0,n})
	\sigma_n(s)ds\nonumber\\
	&\ge \int_{s_0}^{s_0+\frac{1}{\log n}}
	\Pi_{W_{s,n}}(w:\|w - w_0\|_\infty \le D_1\varepsilon_{s_0,n})
	\sigma_n(s)ds.
\end{align}
For $\Lambda_n\in\N$ {to be chosen} below, let $P_{\Lambda_n} w_0$ be the wavelet projection of $w_0 \in B^{s_0}_{11} $ onto the approximation space $V_{\Lambda_n}$. Then, by standard wavelet properties,
\begin{align*}
	\|w_0 - P_{\Lambda_n} w_0\|_\infty
	&\le
	\sum_{l>\Lambda_n}\left\|\sum_{r=1}^{2^{ld}} \langle w_0,\psi_{lr}
	\rangle_2 \psi_{lr} \right\|_\infty\\
	&\lesssim
	\sum_{l>\Lambda_n}2^{l\frac{d}{2}}\sup_{r=1,\dots,2^{ld}}| \langle w_0,\psi_{lr}
	\rangle_2|
	\le
	2^{-\Lambda_n(s_0 - d)}\|w_0\|_{B^{s_0}_{11}}.
\end{align*}
Thus, taking $\Lambda_n\in\N$ such that $2^{\Lambda_n}\simeq n^{s_0/[(2s_0+d)(s_0-d)]}$ (note that this is higher than the usual order $n^{1/(2s_0+d)}$) we have
$$
	\|w_0 - P_{\Lambda_n}w_0\|_\infty
	\lesssim 2^{-\Lambda_n(s_0-d)}
	 \| w_0\|_{B^{s_0}_{11}}
	 \lesssim n^{-\frac{s_0}{2s_0+d}}
	 = \varepsilon_{s_0,n}.
$$
On the other hand, for all $s\in[s_0,s_0+1/\log n]$,
\begin{align*}
	\|P_{\Lambda_n}w_0\|_{B^s_{11}}
	&=\sum_{l\le \Lambda_n}
	2^{l(s-s_0)}2^{l\left(s_0 - \frac{d}{2}\right)}\sum_{r=1}^{2^{ld}}| \langle w_0,\psi_{lr}
	\rangle_2|\\
	&\le
		2^{\Lambda_n(s-s_0)}\|w_0\|_{B^{s_0}_{11}}
	\lesssim
		n^{\frac{s_0}{(2s_0+d)(s_0-d)}(\log n)^{-1}}
	=
		e^{\frac{s_0}{(2s_0+d)(s_0-d)\log n}\log n}
	\lesssim
		1.
\end{align*}
It follows that for all $s\in[s_0,s_0+1/\log n]$, by the triangle inequality (choosing $D_1>0$ above large enough), for some $c_1>0$,
\begin{align*}
	\Pi_{W_{s,n}}(w:\|w - w_0\|_\infty \le D_1\varepsilon_{s_0,n})
	\ge \Pi_{W_{s,n}}(w:\|w - P_{\Lambda_n}w_0\|_\infty \le c_1\varepsilon_{s_0,n}),
\end{align*}
and since $W_{s,n}$ in \eqref{Eq:Wsn} is a fixed {Laplace} random element with associated decentering space equal to $\Zcal_n=B^s_{11} $ and norm $\|\cdot\|_{\Zcal_n} = n\varepsilon_{s,n}^2\|\cdot\|_{B^s_{11}}$ (cf.~Section \ref{Subsec:PropBesovPriors}), by the decentering inequality \eqref{Eq:Decentering} the latter probability is lower bounded by
\begin{align*}
	e^{-\|P_{\Lambda_n}w_0\|_{B^s_{11}} n\varepsilon_{s,n}^2 }
	 &\Pi_{W_{s,n}}\left(w:\|w \|_\infty \le c_1\varepsilon_{s_0,n}\right)\\
	 &\ge
	 e^{-c_2n\varepsilon_{s_0,n}^2 }
	 \Pi_{W_s}\left(w:\|w\|_\infty \le c_1 n \varepsilon_{s_0,n}\varepsilon^2_{s,n}\right).
\end{align*}
By the centred small ball inequality \eqref{Eq:MasterSmallBall} (noting that $W_s$ coincides with $W$ there with the choice $t=s-d>0$),
\begin{align*}
	  \Pi_{W_s}\left(w:\|w\|_\infty \le c_1 n \varepsilon_{s_0,n}\varepsilon^2_{s,n}\right)
	 \ge
		e^{-(c_3 s - c_3d + c_4) \left(c_1 n \varepsilon_{s_0,n}\varepsilon^2_{s,n}\right)^{-d/(s-d)}} 
	\ge
		e^{-c_5n\varepsilon_{s_0,n}^2},
\end{align*}
for $c_3,c_4,c_5>0$, since $s\le s_0+1/\log n\le s_0+1$ for $n$ large enough and
\begin{align*}
	\left(n\varepsilon_{s_0,n}\varepsilon^2_{s,n}\right)^{-\frac{d}{s-d}}
	&=\left(n^\frac{4ss_0 + 2ds_0 + 2ds + d^2 - 2ss_0 - ds_0 - 4ss_0 - 2ds}
	{(2s+d)(2s_0+d)}\right)^{-\frac{d}{s-d}}\\
	&=\left(n^\frac{ s_0d  + d^2 - 2ss_0  }{(2s+d)(2s_0+d)}
	\right)^{-\frac{d}{s-d}}
	=
	\left(n^\frac{d }{2s_0+d}
	\right)^{\frac{2ss_0 - ds_0 - d^2}{(s-d)(2s+d)}}
	\le n\varepsilon_{s_0,n}^2
\end{align*}
as for all $s\ge s_0>d$
\begin{align*}
	(s-d)(2s+d) - (2ss_0 - ds_0 - d^2) 
	&=
	2s(s-s_0) - d(s-s_0) 
	=
	(s-s_0)(2s - d)
	\ge0.
\end{align*}
Combining the previous bounds, we find that for all $s\in[s_0,s_0+1/\log n]$, for sufficiently large $D_1>0$,
$$
	\Pi_{W_{s,n}}(w:\|w - w_0\|_\infty \le D_1\varepsilon_{s_0,n})
	\ge e^{-c_6n\varepsilon_{s_0,n}^2}.
$$
The lower bound \eqref{Eq:HierProb} then implies that the probability of interest is greater than
\begin{align*}
	\int_{s_0}^{s_0+\frac{1}{\log n}}
	e^{-c_6n\varepsilon_{s_0,n}^2}\sigma_n(s)ds
	&=
	e^{-c_6n\varepsilon_{s_0,n}^2}
	\int_{s_0}^{s_0+\frac{1}{\log n}}
		\frac{e^{-n\varepsilon^2_{s,n}}}{\zeta_n}ds.
\end{align*}
Using that $\sigma_n(s)$ is increasing in $s$ and that the normalisation constant satisfies $\zeta_n\simeq \log n$, the integral on the right hand side is bounded below by
$$
		\frac{1/\log n}{\zeta_n} e^{-n\varepsilon^2_{s_0,n}}
	\gtrsim 
		(\log n)^{-2}e^{-n\varepsilon^2_{s_0,n}}
	\ge
		e^{-c_7n\varepsilon^2_{s_0,n}},
$$
for some $c_7>0$. The claim then follows taking $D_2 = c_6 + c_7>0$.

\end{proof}

\begin{lemma}\label{Lem:SievesHier}
Let $\Pi_{W_n}$ be the hierarchical rescaled {Laplace} prior arising as the law of $W_n$ in \eqref{Eq:WnHierarc}. For fixed $s_0>d$, and $M,R>0$, let $s^* = s_0/(1 + M/\log n)$ and define the set
\begin{equation}
\label{Eq:SievesHier}
	\Wcal_n = \left\{w = w^{(1)} +  w^{(2)} , \ \|w^{(1)}\|_ 1
	 \le Rn^{-\frac{s^*}{2s^*+d}},\ \|  w^{(2)}\|_{B^{s^*+d}_{11}}
	\le R n^\frac{d}{2s^*+d}\right\}.
\end{equation}
Then, for all $K>0$, there exist sufficiently large $M,R$, such that for $n\in\N$ large enough,
$$
	\Pi_{W_n}(\Wcal_n^c)\le e^{-Kn^{d/(2s_0+d)}}.
$$
\end{lemma}

\begin{proof}

For any $s>d$, let $\varepsilon_{s,n}$, $\Pi_{W_{s,n}}$, $\Pi_{W_s}$, $W_{s,n}$ and $W_s$ be defined as at the beginning of the proof of Lemma \ref{Lem:SmallBallHier}. For $s^* = s_0/(1 + M/\log n)$,
some algebra shows
\begin{align*}
	\frac{n\varepsilon_{s^*,n}^2}{n\varepsilon_{s_0,n}^2}
	&=n^{\frac{dM+d\log n}{dM + (2s_0+d)\log n} - \frac{d}{2s_0 + d}}
	=n^\frac{2dMs_0}{(2s_0+d)^2\log n+dM(2s_0 + d)}
	= e^{\frac{2dMs_0\log n}{(2s_0+d)^2\log n+dM(2s_0 + d)}},
\end{align*}
which implies, bounding the exponential above and below, that
\begin{align}
\label{Eq:Equiv}
	\sqrt{c_1} n\varepsilon_{s_0,n}^2
	\le n\varepsilon_{s^*,n}^2
	\le c_1n\varepsilon_{s_0,n}^2;
	\quad
	(c_1)^\frac{1}{4}\varepsilon_{s_0,n}
	\le \varepsilon_{s^*,n} 
	\le
	\sqrt{c_1}\varepsilon_{s_0,n},
	\quad c_1:=e^{\frac{2dMs_0}{(2s_0+d)^2}}>1,
\end{align}
the lower bounds holding, for any fixed $M>0$, for all $n\in \N$ large enough. The probability of interest is equal to
\begin{align}
\label{Eq:IntBound1}
	\int_d^{s^*}\Pi_{W_{s,n}}(\Wcal_n^c)\sigma_n(s)ds + 
	\int_{s^*}^{\log n}\Pi_{W_{s,n}}(&\Wcal_n^c)\sigma_n(s)ds\nonumber\\
	&\le 
	e^{-(K+1)n\varepsilon_{s_0,n}^2} + 
	\int_{s^*}^{\log n}\Pi_{W_{s,n}}(\Wcal_n^c)\sigma_n(s)ds,
\end{align}
having chosen $M$  large enough and used that $\sigma_n(s)$ is increasing in $s$, that $\zeta_n\simeq \log n$, and \eqref{Eq:Equiv} to bound the first integral by 
$$
	\int_d^{s^*}\sigma_n(s)ds\le s^* \frac{e^{-n\varepsilon_{s^*,n}^2}}{\zeta_n}
	\le e^{-c_2n\varepsilon_{s^*,n}^2}
	\le e^{-c_2\sqrt{c_1}n\varepsilon_{s_0,n}^2}
	\le e^{-(K+1)n\varepsilon_{s_0,n}^2}.
$$

	We proceed bounding the second integral in \eqref{Eq:IntBound1}. To do so, note \begin{align*}
	&\Pi_{W_{s,n}}(\Wcal_n)\nonumber\\
	&=
	 \Pi_{W_s}\Big(w = w^{(1)} +  w^{(2)}:  \|w^{(1)}\|_ 1
	 \le Rn\varepsilon_{s^*,n}\varepsilon_{s,n}^2,\ \|  w^{(2)}\|_{B^{s^*+d}_{11}}
	\le R n^2\varepsilon^2_{s^*,n}\varepsilon_{s,n}^2\Big).
\end{align*}
The spaces associated to $W_s$ are respectively $\Zcal = B^s_{11} $, with norm $\|\cdot\|_\Zcal = \|\cdot\|_{B^s_{11}}$, and $\Qcal = H^{s-d/2} $, with $\|\cdot\|_\Qcal = \|\cdot\|_{H^{s-d/2}}$ (cf.~Section \ref{Subsec:PropBesovPriors}). Letting
\begin{align*}
	\overline\Wcal_n &= \Big\{\overline w= \overline w^{(1)} 
	+ \overline w^{(2)} + \overline w^{(3)} :  
	 \|\overline w^{(1)}\|_ 1\le n\varepsilon_{s^*,n}\varepsilon_{s,n}^2
	, \ \|\overline w^{(2)}\|_{H^{s-d/2}}\le \sqrt{\overline R n\varepsilon_{s^*,n}^2}, \\
	&\quad\ \|\overline w^{(3)}\|_{B^s_{11}}\le \overline R n\varepsilon_{s^*,n}^2 \Big\},
\end{align*}
the two-level concentration inequality \eqref{Eq:TwoLevConc} implies, using again \eqref{Eq:Equiv}, for $c_3,c_4>0$,
\begin{align*}
	\Pi_{W_s}(\overline\Wcal_n)
	&\ge
		1- \frac{1}{\Pi_{W_s}\left(w: \|w\|_ 1\le n\varepsilon_{s^*,n}
		\varepsilon_{s,n}^2\right)}
		e^{-c_3\overline R n\varepsilon_{s^*,n}^2}\\
	&\ge
		1- \frac{1}{\Pi_{W_s}\left(w: \|w\|_ 1\le n\varepsilon_{s^*,n}
		\varepsilon_{s,n}^2\right)}
		e^{-c_4\overline R n\varepsilon_{s_0,n}^2}.
\end{align*}
As $\|w\|_ 1\le \|w\|_ \infty$, noting that $s\ge s^* = s_0\log n/(M+\log n)>d$ for all $n$ large enough since $s_0>d$, by the centred small ball inequality in \eqref{Eq:MasterSmallBall}, we have for all $s^*<s\le \log n$,
\begin{align*}
	 \Pi_{W_s}\left(w: \|w\|_ 1\le n\varepsilon_{s^*,n}\varepsilon_{s,n}^2\right)
	\ge
		e^{-(c_5 s - c_5 d + c_6)\left(n\varepsilon_{s^*,n}\varepsilon_{s,n}^2\right)^{-d/(s-d)}} 
	\ge
		e^{-c_7\log n\left(n\varepsilon_{s^*,n}\varepsilon_{s,n}^2\right)^{-d/(s-d)}}
\end{align*}
Using \eqref{Eq:Equiv} and the fact that
\begin{align*}
	\log n\left(n\varepsilon_{s^*,n}\varepsilon_{s,n}^2\right)^{-\frac{d}{s-d}}
	&=\log n
	\left(n^\frac{-2ss^* - ds^* + 2ds^* + d^2}{(2s^*+d)(2s+d)}\right)^{-\frac{d}{s-d}}
	=\log n\left(n\varepsilon_{s^*,n}^2\right)^{\frac{2ss^* - ds^* - d^2}{(2s+d)(s-d)}}
	\le n\varepsilon_{s^*,n}^2
\end{align*}
{as} the last exponent is strictly smaller than one,
we obtain
$$
	 \Pi_{W_s}\left(w: \|w\|_ 1\le n\varepsilon_{s^*,n}\varepsilon_{s,n}^2\right)
	 \ge e^{-c_7 n\varepsilon_{s^*,n}^2 }
	 \ge e^{-c_8 n\varepsilon_{s_0,n}^2 }.
$$
For sufficiently large $\overline R>0$, it follows that for all $s\in [s^*,\log n]$
\begin{align}
\label{Eq:ExpIneq1}
	\Pi_{W_s}(\overline\Wcal_n)
	&\ge
		1-
		e^{-(c_4\overline R - c_8) n\varepsilon_{s_0,n}^2}
	\ge
		1-
		e^{-(K+1) n\varepsilon_{s_0,n}^2}.
\end{align}

	Next, approximate $\overline w^{(2)}$ in the definition of $\overline\Wcal_n$ by its wavelet projection $P_{L_n}\overline w^{(2)}$ at resolution $L_n\in \N$ with $2^{L_n}\simeq n^\frac{1}{2s + d}$. Then,
\begin{align*}
	\|\overline w^{(2)} - P_{L_n}\overline w^{(2)}\|_ 1
	&\le 2^{-L_n\left(s - \frac{d}{2}\right)}\|\overline w^{(2)}\|_{H^{s-d/2}}\\
	&\lesssim n^{-\frac{s -d/2}{2s+d}}\sqrt n \varepsilon_{s^*,n}
	=n^\frac{d}{2s+d}\varepsilon_{s^*,n}
	=n\varepsilon_{s,n}^2\varepsilon_{s^*,n}.
\end{align*}
Also, as shown in the conclusion of the proof of Lemma \ref{Lem:SievesFixed},
\begin{align*}
	\| P_{L_n}\overline w^{(2)}\|_{B^s_{11}}
	&\lesssim \sqrt{2^{L_nd}}\|\overline w^{(2)}\|_{H^{s-d/2}}\\
	&\lesssim n^{\frac{d/2}{2s+d}}n^\frac{d/2}{2s^*+d}
	=n^\frac{ds^* + d^2/2 + ds + d^2/2}{(2s+d)(2s^*+d)}
	\le n\varepsilon_{s^*,n}^2
\end{align*}
since the exponent is smaller than $d/(2s^* + d)$ when $s\ge s^*$.
For $\overline w^{(1)}, \overline w^{(3)}$ in the definition of $\overline\Wcal_n$, setting $\widetilde w^{(1)} := \overline w^{(1)} + (\overline w^{(2)} - P_{L_n}\overline w^{(2)})$ and $\widetilde w^{(2)} := \overline w^{(3)} +  P_{L_n}\overline w^{(2)}$ then shows that for all $s\in [s^*,\log n]$ and all $n$ and $\widetilde R$ large enough,
\begin{align*}
	\overline\Wcal_n 
	& 
	\subseteq 
	\widetilde\Wcal_n := 
	\{ \widetilde w = \widetilde w^{(1)} + \widetilde w^{(2)} : 
	 \|\widetilde w^{(1)}\|_ 1\le \widetilde R n\varepsilon_{s,n}^2\varepsilon_{s^*,n}
	,\  \|\widetilde w^{(2)}\|_{B^s_{11}}\le \widetilde R n\varepsilon_{s^*,n}^2 \}.
\end{align*}
In view of \eqref{Eq:ExpIneq1}, 
\begin{align}
\label{Eq:ExpIneq2}
	\Pi_{W_s}(\widetilde \Wcal_n)
	\ge
	1- e^{-(K+1)n\varepsilon_{s_0,n}^2}.
\end{align}
We conclude showing that, choosing sufficiently large $R>0$,
\begin{align}
\label{Eq:DesiredIncl}
	\widetilde \Wcal_n \subseteq \{w = w^{(1)} + w^{(2)}  :  \|w^{(1)}\|_ 1
	 \le Rn\varepsilon_{s,n}^2\varepsilon_{s^*,n},\ \| w^{(2)}\|_{B^{s^*+d}_{11}}
	\le Rn^2\varepsilon_{s,n}^2\varepsilon_{s^*,n}^2\}
\end{align}
for all $s\in [s^*,\log n]$ and all $n\in \N$ large enough. First consider the case $s\in[ s^* + d,\log n]$. Then 
$$
	\| \widetilde w^{(2)}\|_{B^{s^*+d}_{11}}
	\le \| \widetilde w^{(2)}\|_{B^{s}_{11}}
	\le \widetilde R n\varepsilon_{s^*,n}^2
	\le
	\widetilde R n^2\varepsilon_{s,n}^2\varepsilon_{s^*,n}^2
$$
since $n\varepsilon_{s,n}^2\to\infty$. The inclusion \eqref{Eq:DesiredIncl} thus follows with $w^{(1)} = \widetilde w^{(1)},\ w^{(2)} = \widetilde w^{(2)}$, and $R=\widetilde R$. Next consider the range $s\in[s^*, s^* + d)$. Approximate $\widetilde w^{(2)}$ in the definition of $\widetilde \Wcal_n$, by its wavelet projection $P_{\Lambda_n}\widetilde w^{(2)}$ with $\Lambda_n\in \N$ satisfying $2^{\Lambda_n}\simeq n^\frac{d}{(2s + d)(s^*+d-s)}$. Then,
\begin{align*}
	\|P_{\Lambda_n}\widetilde w^{(2)}\|_{B^{s^*+d}_{11}}
	\le
	2^{\Lambda_n(s^*+d - s)}\|\widetilde w^{(2)}\|_{B^s_{11}}
	\lesssim 
	n^\frac{d}{(2s + d)} n\varepsilon_{s^*,n}^2  
	=
	n^2\varepsilon_{s,n}^2\varepsilon_{s^*,n}^2,
\end{align*}
and, using the continuous embedding of $B^0_{11}$ into $L^1$ (e.g., eq.~(21), p.169 in \cite{ST87}),
\begin{align*}
	\|\widetilde w^{(2)} - P_{\Lambda_n}\widetilde w^{(2)}\|_ 1
	&\lesssim\|\widetilde w^{(2)} - P_{\Lambda_n}\widetilde w^{(2)}\|_{B^0_{11}}\\
	&\le 2^{ -\Lambda_n s }\|\widetilde w^{(2)}\|_{B^s_{11}}
	\lesssim
	n^{-\frac{ds}{(2s + d)(s^*+d-s)}}
	n^\frac{d}{2s^*+d}
	= n^\frac{-2ds^2 + ds^* + d^3}{(2s + d)(2s^*+d)(s^*+d-s)}.
\end{align*}
The inclusion \eqref{Eq:DesiredIncl} thus follows showing that the right hand side is smaller than 
$$
	n\varepsilon_{s,n}^2\varepsilon_{s^*,n} 
	= n^{\frac{d}{2s+d}}n^{-\frac{s^*}{2s^*+d}}
	= n^\frac{-2ss^* + ds^* + d^2}{(2s^*+d)(2s+d)}
	=n^\frac{2s^*s^2 + s[-2(s^*)^2 - 3d s^* - d^2] + d(s^*)^2 + 2d^2s^* + d^3}
	{(2s + d)(2s^*+d)(s^*+d-s)}.
$$
Indeed, the difference between the numerators of the exponents equals
\begin{align*}
	\Delta(s)
	&= -2ds^2 + ds^* + d^3
	- 2s^*s^2 - s[-2(s^*)^2 - 3d s^* - d^2] - d(s^*)^2 - 2d^2s^* - d^3\\
	&= -2(s^* + d)s^2
	+ s[2(s^*)^2 + 3d s^* + d^2 ]
	-d(s^*)^2 - 2d^2s^* + ds^*,
\end{align*}
which, as a function of $s$, is a downward-pointing parabola with maximum attained at
$$
	s_v := \frac{2(s^*)^2 + 3ds^* + d^2}{4(s^*+d)} < s^*
$$
since, recalling $s^* = s_0\log n/(M+\log n)>d$ for all $n$ large enough as $s_0>d$,
\begin{align*}
	2(s^*)^2 + 3ds^* + d^2 - 4(s^*+d)s^*
	&= - 2(s^*)^2 - ds^* + d^2\le - 2(s^*)^2<0.
\end{align*}
Hence, since $\Delta(s)$ is decreasing for $s>s_v$, for all $s\in[s^*,s^*+d]$,
\begin{align*}
	\Delta(s)&\le
	\Delta(s^*)\\
	&=
	-2(s^* + d)(s^*)^2
	+ s^*[2(s^*)^2 + 3d s^* + d^2 ]
	-d(s^*)^2 - 2d^2s^* + ds^*\\
	&= -d(d-1)s^*\le 0.
\end{align*}
This shows as required that $\|\widetilde w^{(2)} - P_{\Lambda_n}\widetilde w^{(2)}\|_ 1\lesssim n\varepsilon_{s,n}^2 \varepsilon_{s^*,n}$, so that taking $w^{(1)} := \widetilde w^{(1)} + (\widetilde w^{(2)} - P_{\Lambda_n}\widetilde w^{(2)}) $ and $w^{(2)} := P_{\Lambda_n}\widetilde w^{(2)}$, the desired inclusion \eqref{Eq:DesiredIncl} follows for large enough $R>0$. By \eqref{Eq:ExpIneq2}, we then conclude
\begin{align*}
	\Pi_{W_{s,n}}(\Wcal_n)
	&\ge
	\Pi_{W_s}(\widetilde \Wcal_n)
	\ge
		1 - e^{-(K+1)n\varepsilon_{s_0,n}^2}.
\end{align*}
Combined with \eqref{Eq:IntBound1} this yield
\begin{align*}
	\Pi_{W_n}(\Wcal_n^c)
	&\le e^{-(K+1)n\varepsilon_{s_0,n}^2} + 
	\int_{s^*}^{\log n}e^{-(K+1)n\varepsilon_{s_0,n}^2}\sigma_n(s)ds\\
	&\le 2e^{-(K+1)n\varepsilon_{s_0,n}^2}
	\le e^{-Kn\varepsilon_{s_0,n}^2}.
\end{align*}
\end{proof}

%
%
%
%
%

\subsection{Proof of Theorem \ref{Theo:ContrRateAdaptHomog}}
\label{Subsec:ProofTheoAdaptHomog}

We verify conditions \eqref{Eq:KLCond} - \eqref{Eq:EntropyCond} with $\xi_n := c_1n^{-s_0/(2s_0 + d)}$ for sufficiently large $c_1>0$. The first condition follows for a large enough constant $C>0$ arguing as in the Proof of Theorem \ref{Theo:ContrRateTrunc} and using Lemma \ref{Lem:SmallBallHierTrunc} below. Conditions \eqref{Eq:ApproxSetCond} and \eqref{Eq:EntropyCond} follow exactly as in the proof of Theorem \ref{Theo:ContrRateAdapt}, taking the same sequence of sieves $\Pcal_n = \{\phi_w, \ w\in\Wcal_n\}$ with
$$
	\Wcal_n := \left\{w = w^{(1)} +  w^{(2)} : \|w^{(1)}\|_ 1
	 \le Rn^{-\frac{s^*}{2s^*+d}},\ \|  w^{(2)}\|_{B^{s^*+d}_{11}}
	\le R n^\frac{d}{2s^*+d}\right\},
$$
where $s^*:=s_0/(1 + M/\log n)$ and $R,M>0$ are large enough, using Lemma \ref{Lem:SievesHierTrunc} instead of Lemma \ref{Lem:SievesHier}.

\qed
\endproof

\begin{lemma}\label{Lem:SmallBallHierTrunc}
Let $\Pi_{W_n}$ be the hierarchical non-rescaled truncated Laplace prior arising as the law of $W_n$ in \eqref{Eq:HierTrunc}. Let $w_0\in B^{s_0}_{\infty\infty}(\cube)$, any $s_0>d$. Then, for sufficiently large $D_1, D_2>0$,
$$
	\Pi_{W_n}\left(w:\|w - w_0\|_\infty \le D_1n^{-\frac{s_0}{2s_0+d}}
	\right)\ge e^{-D_2n^{d/(2s_0 + d)}}.
$$
\end{lemma}

\begin{proof}
For each fixed $s>d$, let $\varepsilon_{s,n} := n^{-s/(2s+d)}$ and let $\Pi_{W_{s,n}}$ be the truncated $s$-regular {Laplace} prior arising as the law of
\begin{equation}
\label{Eq:WsnTrunc}
	W_{s,n} 
	=
	\sum_{l=1}^{L_{s,n}} \sum_{r=1}^{2^{ld}} 2^{-l\left(s + \frac{d}{2}\right)}
	W_{lr}\psi_{lr},
	\qquad W_{lr}\iid \textnormal{Laplace},
	\qquad 2^{L_{s,n}} \simeq n^{1/(2s+d)}.
\end{equation}
Then,
$$
 	\Pi_{W_n}\left(w:\|w - w_0\|_\infty \le D_1n^{-\frac{s_0}{2s_0+d}}
	\right)
	\ge\int_{s_0}^{s_0+\frac{1}{\log n}}
	\Pi_{W_{s,n}}(w:\|w - w_0\|_\infty \le D_1\varepsilon_{s_0,n})
	\sigma_n(s)ds.
$$
For all $s\in[s_0,s_0+1/\log n]$, the wavelet projection $P_{L_{s,n}} w_0$ of $w_0 \in B^{s_0}_{\infty\infty} $ satisfies
\begin{align*}
	\|w_0 - P_{L_{s,n}} w_0\|_\infty
	\le
	2^{-L_{s,n}s_0}\|w_0\|_{B^{s_0}_{\infty\infty}}
	\lesssim \varepsilon_{s_0,n},
\end{align*}
having used that
\begin{align*}
	\frac{2^{-L_{s,n}s_0}}{\varepsilon_{s_0,n}}
	&\le  
		n^{- \frac{s_0}{2s_0 + 2/\log n + d} + \frac{s_0}{2s_0 + d}}
	= 
		n^{\frac{ 2s_0/\log n }{(2s_0 + 2/\log n + d)
		(2s_0 + d)}}
	=
		n^{\frac{ 2s_0  }{(2s_0  + d)^2
		\log n + 2}}
	\lesssim 1.
\end{align*}
Proceeding similarly, we also have
\begin{align*}
	\|P_{L_{s,n}}w_0\|_{B^{s+d}_{11}}
	&\le
		2^{L_{s,n}(s-s_0)} 2^{L_{s,n} d }\|w_0\|_{B^{s_0}_{\infty\infty}}
	\lesssim
		n^{\frac{d}{2s_0+d}}
	=
		n \varepsilon_{s_0,n}^2.
\end{align*}
It follows that for some $c_1>0$
\begin{align*}
	\Pi_{W_{s,n}}(w:\|w - w_0\|_\infty \le D_1\varepsilon_{s_0,n})
	\ge \Pi_{W_{s,n}}(w:\|w - P_{L_{s,n}}w_0\|_\infty \le c_1\varepsilon_{s_0,n}),
\end{align*}
and since $W_{s,n}$ in \eqref{Eq:WsnTrunc} is a fixed truncated Laplace random element with associated decentering space $\Zcal_n=V_{L_{s,n}} $ with norm $\|\cdot\|_{\Zcal_n} = \|\cdot\|_{B^{s+d}_{11}}$ (cf.~Section \ref{Subsec:PropBesovPriors}), by the decentering inequality \eqref{Eq:Decentering} the latter probability is lower bounded by
\begin{align*}
	e^{-c_2 n\varepsilon_{s_0,n}^2 }
	 &\Pi_{W_{s,n}}\left(w:\|w \|_\infty \le c_1\varepsilon_{s_0,n}\right).
\end{align*}
By the centred small ball inequality \eqref{Eq:MasterSmallBall} applied with $t = s$,
\begin{align*}
	  \Pi_{W_{s,n}}\left(w:\|w\|_\infty \le c_1  \varepsilon_{s_0,n}\right)
	 \ge
		e^{-(c_3 s + c_4) \left(c_1  \varepsilon_{s_0,n}\right)^
		{-d/s}} 
	\ge
		e^{-c_5n\varepsilon_{s_0,n}^2},
\end{align*}
for $c_3,c_4,c_5>0$ since $s\le s_0+1/\log n\le s_0+1$ for $n$ large enough and
\begin{align*}
	\left(\varepsilon_{s_0,n}\right)^{-\frac{d}{s}}
	&=(n^\frac{d}{2s_0+d})^\frac{s_0}{s}\le n\varepsilon_{s_0,n}^2,
\end{align*}
for all $s\ge s_0$. Combining the previous bounds, we find that for all $s\in[s_0,s_0+1/\log n]$, for sufficiently large $D_1>0$,
$$
	\Pi_{W_{s,n}}(w:\|w - w_0\|_\infty \le D_1\varepsilon_{s_0,n})
	\ge e^{-c_6n\varepsilon_{s_0,n}^2}.
$$
The claim then follows arguing as in the conclusion of the proof of Lemma \ref{Lem:SmallBallHier}.

\end{proof}

\begin{lemma}\label{Lem:SievesHierTrunc}
Let $\Pi_{W_n}$ be the hierarchical non-rescaled truncated Laplace prior arising as the law of $W_n$ in \eqref{Eq:HierTrunc}. For fixed $s_0>d$, and $M,R>0$, let $s^* = s_0/(1 + M/\log n)$ and let $\Wcal_n$ be the sets defined in \eqref{Eq:SievesHier}. Then, for all $K>0$, there exist sufficiently large $M,R$, such that for all $n\in\N$ large enough,
$$
	\Pi_{W_n}(\Wcal_n^c)\le e^{-Kn^{d/(2s_0+d)}}.
$$
\end{lemma}

\begin{proof}
For any $s>d$, let $\varepsilon_{s,n}$, $L_{s,n}$, $\Pi_{W_{s,n}}$ and $W_{s,n}$ be defined as at the beginning of the proof of Lemma \ref{Lem:SmallBallHierTrunc}. Proceeding as in the proof of Lemma \ref{Lem:SievesHier}, we obtain
\begin{align*}
	\Pi_{W_n}(\Wcal_n^c)\le
	e^{-(K+1)n\varepsilon_{s_0,n}^2} + 
	\int_{s^*}^{\log n}\Pi_{W_{s,n}}(\Wcal_n^c)\sigma_n(s)ds.
\end{align*}
To bound the integral on the right hand side, note that the spaces associated to the truncated Laplace random element $W_{s,n}$ are $\Zcal_n = \Qcal_n = V_{L_{s,n}}$, with norms respectively $\|\cdot\|_{\Zcal_n} = \|\cdot\|_{B^{s+d}_{11}}$ and $\|\cdot\|_{\Qcal_n} = \|\cdot\|_{H^{s+d/2}}$. Letting
\begin{align*}
	\overline\Wcal_n &= \Big\{\overline w= \overline w^{(1)} 
	+ \overline w^{(2)} + \overline w^{(3)} : w^{(1)},w^{(2)},w^{(3)}\in V_{L_{s,n}}, \
	 \|\overline w^{(1)}\|_ 1\le \varepsilon_{s^*,n}
	, \ \|\overline w^{(2)}\|_{H^{s+d/2}}\le \sqrt{\overline R n\varepsilon_{s^*,n}^2}, \\
	&\quad\ \|\overline w^{(3)}\|_{B^{s+d}_{11}}\le \overline R n\varepsilon_{s^*,n}^2 \Big\},
\end{align*}
the two-level concentration inequality \eqref{Eq:TwoLevConc} implies, using \eqref{Eq:Equiv}, for $c_3,c_4>0$,
\begin{align*}
	\Pi_{W_{s,n}}(\overline\Wcal_n)
	&\ge
		1- \frac{1}{\Pi_{W_{s,n}}\left(w: \|w\|_ 1\le \varepsilon_{s^*,n}\right)}
		e^{-c_4\overline R n\varepsilon_{s_0,n}^2}.
\end{align*}
As $\|w\|_ 1\le \|w\|_ \infty$ and $\|W_{s,n}\|_\infty\le \|W_{s^*,n}\|_\infty$ for all $s\ge s^*$, we have using again \eqref{Eq:Equiv} and the centred small ball inequality \eqref{Eq:MasterSmallBall} with $t = s^*$,
\begin{align*}
	 \Pi_{W_{s,n}}\left(w: \|w\|_ 1\le \varepsilon_{s^*,n}\right)
	 \ge
	 \Pi_{W_{s^*,n}}\left(w: \|w\|_ \infty\le \varepsilon_{s^*,n}\right)
	\ge
		e^{-c_5\left(\varepsilon_{s^*,n}\right)^{-d/s^*}} 
	\ge
		e^{-c_6 n \varepsilon_{s_0,n}^2}.
\end{align*}
For sufficiently large $\overline R>0$, it follows that
$$
	\Pi_{W_{s,n}}(\overline\Wcal_n)
	\ge
		1-
		e^{-(c_4\overline R - c_6) n\varepsilon_{s_0,n}^2}
	\ge
		1-
		e^{-(K+1) n\varepsilon_{s_0,n}^2}.
$$
Next, note that for $\overline w^{(2)}\in V_{L_{s,n}}$ with $\|\overline w^{(2)}\|_{H^{s+d/2}}\lesssim \sqrt{ n} \varepsilon_{s^*,n}$, by the same computations as in the conclusion of the proof of Lemma \ref{Lem:SievesFixedTrunc},
\begin{align*}
	\| \overline w^{(2)}\|_{B^{s+d}_{11}}
	\le \sqrt{\textnormal{dim}(V_{L_{s,n}})}\|\overline w^{(2)}\|_{H^{s+d/2}}
	\lesssim n^\frac{d/2}{2s+d}  \sqrt{ n} \varepsilon_{s^*,n} \le n \varepsilon_{s^*,n}^2 
\end{align*}
since $s\ge s^*$. Taking $w^{(1)} = \overline w^{(1)}$, $w^{(2)} = \overline w^{(2)} + \overline w^{(3)}$ and using the embedding $B^{s+d}_{11}\subset B^{s^*+d}_{11}$ then shows that $\overline\Wcal_n \subseteq \Wcal_n$ for sufficiently large $R$,
whence
$$
	\Pi_{W_{s,n}}( \Wcal_n)
	\ge
	\Pi_{W_{s,n}}(\overline \Wcal_n)
	\ge
	1- e^{-(K+1)n\varepsilon_{s_0,n}^2},
$$
for all $s\ge s^*$. The proof then carries over as in the conclusion of the proof of Lemma \ref{Lem:SievesHier}.

\end{proof}

\appendix

\section{Additional material}
\label{Sec:AdditionalMaterial}

%
%
%
%
%

\subsection{General properties of {Laplace} priors}
\label{Subsec:PropBesovPriors}

In this section we record, for ease of exposition, a number of properties of {Laplace} priors employed throughout the paper, largely based on the results of Agapiou et al.~\cite{ADH21}. For $t>0$, consider a (possibly $n$-independent) $t$-regular {Laplace} prior $\Pi_{W_n}$ on $C(\cube)$, arising as the law of 
$$
	W_n 
	= \sum_{l=1}^\infty \sum_{r=1}^{2^{ld}} \sigma_{n,lr}W_{lr}\psi_{lr}, 
	\qquad W_{lr}\iid \textnormal{Laplace},
$$
with $\sigma_{n,lr}>0$ satisfying \eqref{Eq:Scaling}. An analogous argument as in Lemma 5.2 and Proposition 6.1 in \cite{ADH21} (see also Lemma 7.1 in \cite{AW23}) shows that $W_n\in C(\cube)\cap B^{t'}_{rr}$ almost surely for all $t'<t$ and $r\in[1,\infty]$. On the contrary, $\Pr(W_n\in B^t_{rr}) = 0$. As $\Pi_{W_n}$ is supported on $C(\cube)$, its log-concavity (cf.~Lemma 3.4 in \cite{ABDH18}) implies (via a Fernique-like theorem \cite[Section 2]{DHS12} and the exponential Markov inequality) the following sup-norm concentration inequality: for some constants $a_1,a_2>0$,
\begin{equation}
\label{Eq:SupNormConc}
	\Pr\left(\|W_n\|_\infty > R\right) \le a_1 e^{-a_2 R}, \qquad \textnormal{all}\ R>0.
\end{equation}

 %
 
 	The finer information geometry properties of $\Pi_{W_n}$ are characterised by two associated function spaces,
$$
	\Qcal_n := \left \{w= \sum_{l=1}^\infty\sum_{r = 1}^{2^{ld}}w_{lr}\psi_{lr} :  
	\|w\|_{\Qcal_n}^2:=\sum_{l=1}^\infty\sum_{r = 1}^{2^{ld}}\sigma_{n,lr}^{-2}
	|w_{lr}|^2<\infty
	\right\}, 
$$
and
$$
	\Zcal_n := \left \{w= \sum_{l=1}^\infty\sum_{r = 1}^{2^{ld}}w_{lr}\psi_{lr} :  
	\|w\|_{\Zcal_n}:=\sum_{l=1}^\infty\sum_{r = 1}^{2^{ld}}\sigma_{n,lr}^{-1}|w_{lr}|<\infty
	\right\}.
$$
Note that $\Zcal_n\subset\Qcal_n$ with continuous embedding.
The weighted $\ell^2$-space $\Qcal_n$ contains the admissible shifts $w$ for which the law of the random function $W_n + w$ is absolutely continuous with respect to $\Pi_{W_n}$ (cf.~\cite[Proposition 2.7]{ADH21}). On the other hand, the weighted $\ell^1$-norm $\|\cdot\|_{\Zcal_n}$ quantifies the loss in prior probability of non-centred balls compared to centred ones: by Proposition 2.11 in \cite{ADH21}, for all $w\in\Zcal_n$, all $\xi>0$,
\begin{equation}
\label{Eq:Decentering}
	\Pr\left( \|W_n - w\|_\infty \le \xi  \right) \ge e^{-\|w\|_{\Zcal_n}}
	\Pr\left( \|W_n \|_\infty \le \xi  \right).
\end{equation}
In the proofs we often refer to $\Zcal_n$ as the `decentering' space. Via the two-level concentration inequality in Proposition 2.15 in \cite{ADH21}, the bulk of the prior probability mass is seen to be contained in an enlargement of the sum of sufficiently large balls in $\Qcal_n$ and $\Zcal_n$: for some constant $a_3>0$, for all Borel measurable $A\subseteq C(\cube)$ and all $R>0$,
\begin{align}
\label{Eq:TwoLevConc}
	\Pr\Big( W_n= W^{(1)}_n + W^{(2)}_n + W^{(3)}_n 
	 :  W^{(1)}_n\in A, \ \|W^{(2)}_n\|_{\Qcal_n}\le &\sqrt R, \ 
	\|W^{(3)}_n\|_{\Zcal_n}\le R  \Big)\nonumber\\
	&\ge 1 - \frac{1}{\Pr(W_n\in A)}e^{-R/a_3}.
\end{align}

	Finally, Proposition 6.3 in \cite{ADH21} provides a lower bound for the decay of the `centred small ball probability' appearing in the right hand side of \eqref{Eq:Decentering} as the radius $\xi\to0$. Below, we slightly reformulate such estimate, extending it to the multi-dimensional case and keeping track of how the multiplicative constant appearing in the statement of Proposition 6.3 in \cite{ADH21} depends on the prior regularity parameter. Inspection of the proof of that result (and of the proof of Lemma 2.1 in \cite{S96}) shows that for $\sigma_{n,lr} = 2^{-l(t + d/2)}$, any $t>0$, the ($n$-independent) $t$-regular Laplace random element
$$
	W
	= \sum_{l=1}^\infty \sum_{r=1}^{2^{ld}} 2^{-l(t + d/2)}W_{lr}\psi_{lr}, 
	\qquad W_{lr}\iid \textnormal{Laplace},
$$
satisfies as $\xi\to0$,
\begin{equation}
\label{Eq:MasterSmallBall}
	\Pr\left( \|W \|_\infty \le \xi  \right) \ge e^{-(a_4 t + a_5)\xi^{-d/t}}
\end{equation}
for some constants $a_4, a_5>0$.


%
%
%
%
%

\subsection{Auxiliary results}
\label{Subsec:AuxResults}

In this section we collect two auxiliary results used in the proofs of the main theorems.

\begin{lemma}\label{Lem:InfoIneq}
	Let $\phi:\R\to(0,\infty)$ be a strictly increasing and continuous function with uniformly Lipschitz logarithm with Lipschitz constant $\Lcal>0$. For $w,w'\in C(\cube)$, let $\phi_w,\phi_{w'}$ be the associated probability density functions defined according to \eqref{Eq:PhiWn}. Then,
$$
	\|\phi_w - \phi_{w'}\|_1\le \frac{2\Lcal e^{\Lcal\|w - w'\|_\infty}}{\phi(-\|w'\|_\infty)}
	\|w - w'\|_1.
$$
\end{lemma}

\begin{proof}
Some algebra yields
\begin{align*}
	\|\phi_w - \phi_{w'}\|_1
	&=
		\left\| \frac{\phi\circ w}{\|\phi\circ w\|_1} 
		- \frac{\phi\circ w'}{\|\phi\circ w'\|_1} \right\|_1\\
	&\le 
		\frac{2}{\|\phi\circ w'\|_1}\|\phi\circ w - \phi\circ w'\|_1 
	\le
		\frac{2}{\phi(-\|w'\|_\infty)}\|\phi\circ w - \phi\circ w'\|_1.
\end{align*}
The latter norm equals
\begin{align*}
	\int_{\cube}|\phi(w(x)) &- \phi(w'(x))|dx\\
	&=
		\int_{\cube}\left| e^{\log\frac{\phi(w(x))}{\phi(w'(x))}} - 1\right|dx\\
	&\le
		\int_{\cube}
		\left|\log\phi(w(x)) - \log\phi(w'(x))\right|  e^{|\log\phi(w(x)) - \log\phi(w'(x))|}dx
\end{align*}
having used that for all $z\in\R$, $|e^z - 1|\le |z|e^{|z|}$. Recalling that $\log\phi$ is uniformly Lipschitz with Lipschitz constant $\Lcal>0$, the claim follows upper bounding the integral in the last line by
\begin{align*}
	\Lcal\int_{\cube}|w(x) - w'(x)| e^{\Lcal|w(x) - w'(x)|}dx
	\le \Lcal e^{L\|w - w'\|_\infty}\|w - w'\|_1.
\end{align*}
\end{proof}

\begin{lemma}\label{Lem:InfoIneqLip}
	For fixed $B>0$, Let $\phi:\R\to(B,\infty)$ be a strictly increasing and uniformly Lipschitz function with Lipschitz constant $\Lcal>0$. For $w, w'\in C(\cube)$, let $\phi_w,\phi_{w'}$ be the associated probability density functions defined according to \eqref{Eq:PhiWn}. Then,
\begin{enumerate}
\item
$$
	\max\left\{-E_{\phi_{w'}}\left(\log \frac{\phi_w}{\phi_{w'}}(X)\right), 
	E_{\phi_{w'}}\left(\log \frac{\phi_w}{\phi_{w'}}(X)\right)^2\right\}
	\lesssim
	\frac{\Lcal^2}{B^2} \left\|\frac{\phi_{w'}}{\phi_w}\right\|_\infty
	\left\|  w - w'\right\|_2^2;
$$
\item
$$
	\|\phi_w - \phi_{w'}\|_1\le \frac{2\Lcal}{B}\| w - w'\|_1.
$$
\end{enumerate}
\end{lemma}

\begin{proof}
For point 1., by Lemma B.2 in \cite{GvdV17},
\begin{align*}
	\max\left\{-E_{\phi_{w'}}\left(\log \frac{\phi_w}{\phi_{w'}}(X)\right), 
	E_{\phi_{w'}}\left(\log \frac{\phi_w}{\phi_{w'}}(X)\right)^2\right\}
	&\lesssim  \left\|\frac{\phi_{w'}}{\phi_w}\right\|_\infty d^2_H(\phi_w',\phi_w),
\end{align*}
where $d_H$ is the Hellinger distance (cf.~\eqref{Eq:HellDist}). Using that
$$
	d_H(\phi_{w'},\phi_w)
	=\left\|\frac{\sqrt{\phi\circ w'}}{\|\sqrt{\phi\circ w'}\|_2}
	-\frac{\sqrt{\phi\circ w}}{\|\sqrt{\phi\circ w}\|_2}
	 \right\|_2
	 \le \frac{2}{\|\sqrt{\phi\circ w' }\|_2}\left\| \sqrt{\phi\circ w'} 
	 - \sqrt{\phi\circ w}\right\|_2,
$$
and that
\begin{align*}
	\left\| \sqrt{\phi\circ w'} - \sqrt{\phi\circ w}\right\|_2
	&=\left\| \frac{\phi\circ w' - \phi\circ w}{\sqrt{\phi\circ w'} + \sqrt{\phi\circ w}}\right\|_2\\
	&\le \left\|\frac{1}{\sqrt{\phi\circ w'} + \sqrt{\phi\circ w}}\right\|_\infty
	\|\phi\circ w - \phi\circ w'\|_2\\
	&\le \frac{1}{\sqrt{\phi(-\|w'\|_\infty)} + \sqrt{\phi(-\|w\|_\infty)}}
	\|\phi\circ w - \phi\circ w'\|_2
\end{align*}
the claim follows since $\phi(z)>B$ for all $z\in\R$ and $\|\phi\circ w - \phi\circ w'\|_2\le \Lcal\|w - w'\|_2$. For point 2., arguing as in the proof of Lemma \ref{Lem:InfoIneq},
\begin{align*}
	\|\phi_w - \phi_{w'}\|_1
	\le
		\frac{2}{\phi(-\|w'\|_\infty)}\|\phi\circ w - \phi\circ w'\|_1,
\end{align*}
whence the claim follows since $\phi$ is bounded below by $B$ and uniformly Lipschitz.
\end{proof}

\paragraph{Acknowledgement.} We would like to thank Kolyan Ray, Judith Rousseau and Sergios Agapiou for valuable discussions. We are also grateful to the Associate Editor and to two anonymous Referees for very helpful comments that lead to an improvement of the paper.  This research has been partially supported by MUR, PRIN project 2022CLTYP4. The first version of the manuscript was completed while M.G.~was affiliated with the University of Oxford and supported by the ERC grant agreement No.~834275 (GTBB).

\bibliographystyle{acm}

\bibliography{References.bib}

\begin{thebibliography}{10}

\bibitem{AN19}
{\sc Abraham, K., and Nickl, R.}
\newblock On statistical {C}alder\'{o}n problems.
\newblock {\em Math. Stat. Learn. 2}, 2 (2019), 165--216.

\bibitem{ABDH18}
{\sc Agapiou, S., Burger, M., Dashti, M., and Helin, T.}
\newblock Sparsity-promoting and edge-preserving maximum {\it a posteriori}
  estimators in non-parametric {B}ayesian inverse problems.
\newblock {\em Inverse Problems 34}, 4 (2018), 045002, 37.

\bibitem{ADH21}
{\sc Agapiou, S., Dashti, M., and Helin, T.}
\newblock {Rates of contraction of posterior distributions based on
  p-exponential priors}.
\newblock {\em Bernoulli 27}, 3 (2021), 1616 -- 1642.

\bibitem{AS22}
{\sc Agapiou, S., and Savva, A.}
\newblock Bayesian adaptive inference based on $p$-exponential priors.
\newblock {\em arXiv preprint arXiv:2209.06045\/} (2022).

\bibitem{AW23}
{\sc Agapiou, S., and Wang, S.}
\newblock Laplace priors and spatial inhomogeneity in {B}ayesian inverse
  problems.
\newblock {\em Bernoulli, \textnormal{to appear}\/}.

\bibitem{AGR13}
{\sc Arbel, J., Gayraud, G., and Rousseau, J.}
\newblock Bayesian optimal adaptive estimation using a sieve prior.
\newblock {\em Scand. J. Stat. 40}, 3 (2013), 549--570.

\bibitem{BD06}
{\sc Bioucas-Dias, J.~M.}
\newblock Bayesian wavelet-based image deconvolution: a {GEM} algorithm
  exploiting a class of heavy-tailed priors.
\newblock {\em IEEE Trans. Image Process. 15}, 4 (2006), 937--951.

\bibitem{BS10}
{\sc Bourdaud, G., and Sickel, W.}
\newblock Composition operators on function spaces with fractional order of
  smoothness.
\newblock In {\em Harmonic analysis and nonlinear partial differential
  equations}, RIMS K\^{o}ky\^{u}roku Bessatsu, B26. Res. Inst. Math. Sci.
  (RIMS), Kyoto, 2011, pp.~93--132.

\bibitem{BG15}
{\sc Bui-Thanh, T., and Ghattas, O.}
\newblock A scalable algorithm for {MAP} estimators in {B}ayesian inverse
  problems with {B}esov priors.
\newblock {\em Inverse Probl. Imaging 9}, 1 (2015), 27--53.

\bibitem{Cast08}
{\sc Castillo, I.}
\newblock Lower bounds for posterior rates with {G}aussian process priors.
\newblock {\em Electron. J. Stat. 2\/} (2008), 1281--1299.

\bibitem{CN13}
{\sc Castillo, I., and Nickl, R.}
\newblock Nonparametric {B}ernstein--von {M}ises {T}heorems in {G}aussian white
  noise.
\newblock {\em Ann. Statist. 41}, 4 (2013), 1999--2028.

\bibitem{C08}
{\sc Cavalier, L.}
\newblock Nonparametric statistical inverse problems.
\newblock {\em Inverse Problems}, 24 (2008).

\bibitem{CDPS18}
{\sc Chen, V., Dunlop, M.~M., Papaspiliopoulos, O., and Stuart, A.~M.}
\newblock Dimension-robust {MCMC} in {B}ayesian inverse problems.
\newblock {\em arXiv preprint arXiv:1803.03344\/} (2018).

\bibitem{CDPS19}
{\sc Chen, V., Dunlop, M.~M., Papaspiliopoulos, O., and Stuart, A.~M.}
\newblock Dimension-robust mcmc in bayesian inverse problems, 2019.

\bibitem{DHS12}
{\sc Dashti, M., Harris, S., and Stuart, A.~M.}
\newblock Besov priors for {B}ayesian inverse problems.
\newblock {\em Inverse Probl. Imaging 6}, 2 (2012), 183--200.

\bibitem{DS17}
{\sc Dashti, M., and Stuart, A.~M.}
\newblock The {B}ayesian approach to inverse problems.
\newblock In {\em Handbook of uncertainty quantification. {V}ol. 1, 2, 3}.
  Springer, Cham, 2017, pp.~311--428.

\bibitem{DJ98}
{\sc Donoho, D.~L., and Johnstone, I.~M.}
\newblock Minimax estimation via wavelet shrinkage.
\newblock {\em Ann. Statist. 26}, 3 (1998), 879--921.

\bibitem{GGvdV00}
{\sc Ghosal, S., Ghosh, J.~K., and van~der Vaart, A.~W.}
\newblock Convergence rates of posterior distributions.
\newblock {\em Ann. Statist. 28}, 2 (2000), 500--531.

\bibitem{GLvdV08}
{\sc Ghosal, S., Lember, J., and van~der Vaart, A.}
\newblock Nonparametric {B}ayesian model selection and averaging.
\newblock {\em Electron. J. Stat. 2\/} (2008), 63--89.

\bibitem{GvdV07}
{\sc Ghosal, S., and van~der Vaart, A.}
\newblock Convergence rates of posterior distributions for non-i.i.d.
  observations.
\newblock {\em Ann. Statist. 35}, 1 (2007), 192--223.

\bibitem{GvdV17}
{\sc Ghosal, S., and van~der Vaart, A.~W.}
\newblock {\em Fundamentals of Nonparametric Bayesian Inference}.
\newblock Cambridge University Press, New York, 2017.

\bibitem{GN11}
{\sc Gin{\'e}, E., and Nickl, R.}
\newblock Rates of contraction for posterior distributions in {$L^r$}-metrics,
  {$1\leq r\leq\infty$}.
\newblock {\em Ann. Statist. 39}, 6 (2011), 2883--2911.

\bibitem{GN16}
{\sc Gin\'e, E., and Nickl, R.}
\newblock {\em Mathematical foundations of infinite-dimensional statistical
  models}.
\newblock Cambridge University Press, New York, 2016.

\bibitem{GN20}
{\sc Giordano, M., and Nickl, R.}
\newblock Consistency of {B}ayesian inference with {G}aussian process priors in
  an elliptic inverse problem.
\newblock {\em Inverse Problems 36}, 8 (2020), 085001--85036.

\bibitem{GR22}
{\sc Giordano, M., and Ray, K.}
\newblock Nonparametric {B}ayesian inference for reversible multidimensional
  diffusions.
\newblock {\em Ann. Statist. 50}, 5 (2022), 2872--2898.

\bibitem{GRSH22}
{\sc Giordano, M., Ray, K., and Schmidt-Hieber, J.}
\newblock On the inability of gaussian process regression to optimally learn
  compositional functions.
\newblock In {\em Advances in Neural Information Processing Systems\/} (2022),
  S.~Koyejo, S.~Mohamed, A.~Agarwal, D.~Belgrave, K.~Cho, and A.~Oh, Eds.,
  vol.~35, Curran Associates, Inc., pp.~22341--22353.

\bibitem{HKPT98}
{\sc H\"{a}rdle, W., Kerkyacharian, G., Picard, D., and Tsybakov, A.}
\newblock {\em Wavelets, approximation, and statistical applications}, vol.~129
  of {\em Lecture Notes in Statistics}.
\newblock Springer-Verlag, New York, 1998.

\bibitem{HB15}
{\sc Helin, T., and Burger, M.}
\newblock Maximum a posteriori probability estimates in infinite-dimensional
  {B}ayesian inverse problems.
\newblock {\em Inverse Problems 31}, 8 (2015), 085009, 22.

\bibitem{H19}
{\sc Hosseini, B.}
\newblock Two metropolis--hastings algorithms for posterior measures with
  non-gaussian priors in infinite dimensions.
\newblock {\em SIAM/ASA Journal on Uncertainty Quantification 7}, 4 (2019),
  1185--1223.

\bibitem{JPG16}
{\sc Jia, J., Peng, J., and Gao, J.}
\newblock Bayesian approach to inverse problems for functions with a
  variable-index {B}esov prior.
\newblock {\em Inverse Problems 32}, 8 (2016), 085006, 32.

\bibitem{K22}
{\sc Kekkonen, H.}
\newblock Consistency of {B}ayesian inference with {G}aussian process priors
  for a parabolic inverse problem.
\newblock {\em Inverse Problems 38}, 3 (2022), 035002, 29.

\bibitem{KLSS23}
{\sc Kekkonen, H., Lassas, M., Saksman, E., and Siltanen, S.}
\newblock Random tree {B}esov priors---towards fractal imaging.
\newblock {\em Inverse Probl. Imaging 17}, 2 (2023), 507--531.

\bibitem{KSvdVvZ15}
{\sc Knapik, B., Szab{\`o}, B., van~der Vaart, A.~W., and van Zanten, H.}
\newblock Bayes procedures for adaptive inference in inverse problems for the
  white noise model.
\newblock {\em Probab. Theory Relat. Fields}, 164 (2015), 771--813.

\bibitem{KvdVvZ11}
{\sc Knapik, B., van~der Vaart, A.~W., and van Zanten, J.~H.}
\newblock Bayesian inverse problems with {G}aussian priors.
\newblock {\em Ann. Statist. 39}, 5 (2011), 2626--2657.

\bibitem{KLNS12}
{\sc Kolehmainen, V., Lassas, M., Niinim\"{a}ki, K., and Siltanen, S.}
\newblock Sparsity-promoting {B}ayesian inversion.
\newblock {\em Inverse Problems 28}, 2 (2012), 025005, 28.

\bibitem{L12}
{\sc Lasanen, S.}
\newblock Non-{G}aussian statistical inverse problems. {P}art {I}: {P}osterior
  distributions.
\newblock {\em Inverse Probl. Imaging 6}, 2 (2012), 215--266.

\bibitem{LSS09}
{\sc Lassas, M., Saksman, E., and Siltanen, S.}
\newblock Discretization-invariant {B}ayesian inversion and {B}esov space
  priors.
\newblock {\em Inverse Probl. Imaging 3}, 1 (2009), 87--122.

\bibitem{LS04}
{\sc Lassas, M., and Siltanen, S.}
\newblock Can one use total variation prior for edge-preserving {B}ayesian
  inversion?
\newblock {\em Inverse Problems 20}, 5 (2004), 1537--1563.

\bibitem{LvdV07}
{\sc Lember, J., and van~der Vaart, A.}
\newblock On universal {B}ayesian adaptation.
\newblock {\em Statist. Decisions 25}, 2 (2007), 127--152.

\bibitem{LP01}
{\sc Leporini, D., and Pesquet, J.-C.}
\newblock Bayesian wavelet denoising: Besov priors and non-{G}aussian noises.
\newblock {\em Signal Processing 81}, 1 (2001), 55--67.
\newblock Special section on Markov Chain Monte Carlo (MCMC) Methods for Signal
  Processing.

\bibitem{MNP21a}
{\sc Monard, F., Nickl, R., and Paternain, G.~P.}
\newblock Consistent inversion of noisy non-{A}belian {X}-ray transforms.
\newblock {\em Comm. Pure Appl. Math. 74}, 5 (2021), 1045--1099.

\bibitem{NvdGW20}
{\sc Nickl, R., van~de Geer, S., and Wang, S.}
\newblock Convergence rates for penalized least squares estimators in {PDE}
  constrained regression problems.
\newblock {\em SIAM/ASA J. Uncertain. Quantif. 8}, 1 (2020), 374--413.

\bibitem{NW22}
{\sc Nickl, R., and Wang, S.}
\newblock On polynomial-time computation of high-dimensional posterior measures
  by langevin-type algorithms.
\newblock {\em Journal of the European Mathematical Society, \textnormal{to
  appear}\/}.

\bibitem{NSK07}
{\sc Niinim{\"a}ki, K., Siltanen, S., and Kolehmainen, V.}
\newblock Bayesian multiresolution method for local tomography in dental
  {X}-ray imaging.
\newblock {\em Phys. Med. Biol. 22}, 52 (2007), 6663--78.

\bibitem{RVJKLMS06}
{\sc Rantala, M., V{\"a}nsk{\"a}, S., J{\"a}rvenp{\"a}{\"a}, S., Kalke, M.,
  Lassas, M., Moberg, J., and Siltanen, S.}
\newblock Wavelet-based reconstruction for limited-angle {X}-ray tomography.
\newblock {\em IEEE Transactions on Medical Imaging 25}, 2 (2006), 210--217.

\bibitem{R13}
{\sc Ray, K.}
\newblock Bayesian inverse problems with non-conjugate priors.
\newblock {\em Electron. J. Stat. 7\/} (2013), 2516--2549.

\bibitem{RR21}
{\sc Rockova, V., and Rousseau, J.}
\newblock Ideal {B}ayesian spatial adaptation.
\newblock {\em arXiv preprint arXiv:2105.12793\/} (2021).

\bibitem{ROF92}
{\sc Rudin, L.~I., Osher, S., and Fatemi, E.}
\newblock Nonlinear total variation based noise removal algorithms.
\newblock {\em Physica D: Nonlinear Phenomena 60}, 1 (1992), 259--268.

\bibitem{SE15}
{\sc Sakhaee, E., and Entezari, A.}
\newblock {Spline-based sparse tomographic reconstruction with Besov priors}.
\newblock In {\em Medical Imaging 2015: Image Processing\/} (2015), S.~Ourselin
  and M.~A. Styner, Eds., vol.~9413, International Society for Optics and
  Photonics, SPIE, pp.~101 -- 108.

\bibitem{ST87}
{\sc Schmeisser, H.-J., and Triebel, H.}
\newblock {\em Topics in {F}ourier analysis and function spaces}.
\newblock A Wiley-Interscience Publication. John Wiley \& Sons, Ltd.,
  Chichester, 1987.

\bibitem{SS23}
{\sc Schwab, C., and Stein, A.}
\newblock Multilevel monte carlo fem for elliptic pdes with besov random tree
  priors, 2023.

\bibitem{SW01}
{\sc Shen, X., and Wasserman, L.}
\newblock Rates of convergence of posterior distributions.
\newblock {\em Ann. Statist. 29}, 3 (2001), 687--714.

\bibitem{S96}
{\sc Stolz, W.}
\newblock Some small ball probabilities for {G}aussian processes under
  nonuniform norms.
\newblock {\em J. Theoret. Probab. 9}, 3 (1996), 613--630.

\bibitem{T83}
{\sc Triebel, H.}
\newblock {\em Theory of function spaces}, vol.~78 of {\em Monographs in
  Mathematics}.
\newblock Birkh\"auser Verlag, Basel, 1983.

\bibitem{vdVvZ08}
{\sc van~der Vaart, A.~W., and van Zanten, J.~H.}
\newblock Rates of contraction of posterior distributions based on {G}aussian
  process priors.
\newblock {\em Ann. Statist. 36}, 3 (2008), 1435--1463.

\bibitem{vWvZ16}
{\sc van Waaij, J., and van Zanten, H.}
\newblock {G}aussian process methods for one-dimensional diffusions: {O}ptimal
  rates and adaptation.
\newblock {\em Electron. J. Stat. 10}, 1 (2016), 628--645.

\bibitem{VLS09}
{\sc V\"{a}nsk\"{a}, S., Lassas, M., and Siltanen, S.}
\newblock Statistical {X}-ray tomography using empirical {B}esov priors.
\newblock {\em Int. J. Tomogr. Stat. 11}, S09 (2009), 3--32.

\bibitem{WBSCM17}
{\sc Wang, Z., Bardsley, J.~M., Solonen, A., Cui, T., and Marzouk, Y.~M.}
\newblock Bayesian inverse problems with {$l_1$} priors: a
  randomize-then-optimize approach.
\newblock {\em SIAM J. Sci. Comput. 39}, 5 (2017), S140--S166.

\end{thebibliography}

\end{document}